\documentclass[11pt]{amsart}
\usepackage{amsmath,amssymb,amsthm,amsxtra} 
\usepackage{stmaryrd}
\usepackage{times}
\usepackage[margin=1.5in]{geometry}
\usepackage[linktocpage=true, hidelinks, colorlinks=true, allcolors=blue]{hyperref}
\usepackage{comment}

\makeatletter
\def\l@subsection{\@tocline{2}{0pt}{2.5pc}{5pc}{}}
\renewcommand\tocchapter[3]{%
  \indentlabel{\@ifnotempty{#2}{\ignorespaces#2.\quad}}#3%
}
\newcommand\@dotsep{4.5}
\def\@tocline#1#2#3#4#5#6#7{\relax
  \ifnum #1>\c@tocdepth 
  \else
    \par \addpenalty\@secpenalty\addvspace{#2}%
    \begingroup \hyphenpenalty\@M
    \@ifempty{#4}{%
      \@tempdima\csname r@tocindent\number#1\endcsname\relax
    }{%
      \@tempdima#4\relax
    }%
    \parindent\z@ \leftskip#3\relax \advance\leftskip\@tempdima\relax
    \rightskip\@pnumwidth plus1em \parfillskip-\@pnumwidth
    #5\leavevmode\hskip-\@tempdima{#6}\nobreak
    \leaders\hbox{$\m@th\mkern \@dotsep mu\hbox{.}\mkern \@dotsep mu$}\hfill
    \nobreak
    \hbox to\@pnumwidth{\@tocpagenum{#7}}\par
    \nobreak
    \endgroup
  \fi}
\makeatother
\AtBeginDocument{%
\makeatletter
\expandafter\renewcommand\csname r@tocindent0\endcsname{0pt}
\makeatother
}
\def\l@subsection{\@tocline{2}{0pt}{2.5pc}{5pc}{}}

\newtheorem{thm}{Theorem}[section]
\newtheorem{lemma}[thm]{Lemma}
\newtheorem{proposition}[thm]{Proposition}

\newtheorem{definition}[thm]{Definition}
\newtheorem{corollary}[thm]{Corollary}

\newtheorem{question}[thm]{Question}

\newtheorem*{maintheorem*}{Main Theorem}
\newtheorem*{theorem*}{Theorem}
\newtheorem*{corollary*}{Corollary}

\newcommand{\p}{\mathbb{P}}

\newcommand{\dom}{\mathrm{dom}}
\newcommand{\ran}{\mathrm{ran}}

\newcommand{\Succ}{\mathrm{Succ}}
\newcommand{\ISucc}{\mathrm{ISucc}}

\newcommand{\h}{\mathrm{ht}}

\newcommand{\res}{\upharpoonright}

\newcommand{\ka}{\kappa}

\begin{document}

\title[An almost Kurepa Suslin tree with strongly non-saturated square]{An almost Kurepa Suslin tree with strongly \\ non-saturated square}

\author{John Krueger and Eduardo Martinez Mendoza}

\address{John Krueger, Department of Mathematics, 
	University of North Texas,
	1155 Union Circle \#311430,
	Denton, TX 76203, USA}
\email{john.krueger@unt.edu}

\address{Eduardo Martinez Mendoza, Department of Mathematics, 
	University of North Texas,
	1155 Union Circle \#311430,
	Denton, TX 76203, USA}
\email{eduardomartinezmendoza@my.unt.edu}

\date{June 14, 2024; revised September 6, 2025}



\begin{abstract}
	For uncountable downwards closed subtrees $U$ and $W$ of 
	an $\omega_1$-tree $T$, 
	we say that $U$ and $W$ are \emph{strongly almost disjoint} 
	if their intersection is a finite union of countable chains. 
	The tree $T$ is \emph{strongly non-saturated} if there exists a strongly 
	almost disjoint family of $\omega_2$-many uncountable downwards closed subtrees of $T$. 
	In this article we construct a Knaster forcing which adds a Suslin tree together 
	with a family of $\omega_2$-many strongly almost disjoint automorphisms of it 
	(and thus the square of the Suslin tree is strongly non-saturated). 
	To achieve this goal, we introduce a new idea called \emph{$\rho$-separation}, 
	which is an adaptation to the finite context of the notion of separation 
	which was recently introduced by Stejskalov\'{a} and the first author 
	for the purpose of adding automorphisms of a tree with a forcing 
	with countable conditions.
\end{abstract}

\maketitle

\tableofcontents

\section{Introduction}

The consistency of the existence of a Suslin tree was originally proven 
independently by Jech \cite{jech67} and Tennenbaum \cite{tennenbaum} 
using the technique of forcing. 
In Jech's forcing, conditions are countable initial segments of the 
generic tree with a top level, whereas Tennenbaum's forcing consists 
of finite approximations of the generic tree. 
Jech's forcing is countably closed and Tennenbaum's forcing is c.c.c. 
As a variation of his forcing for adding a Suslin tree, Jech \cite{jech72} 
defined a countably closed forcing which adds a Suslin tree 
together with $\ka$-many automorphisms of it, 
where $\ka$ is any infinite cardinal number satisfying that $\kappa^\omega = \kappa$. 
While it is not mentioned explicitly in his article, the automorphisms added by 
Jech's forcing are almost disjoint in the sense that any two of them agree on 
only countably many elements of the tree.  
As a consequence of this fact, 
if $\ka \ge \omega_2$ then 
the generic Suslin tree is an \emph{almost Kurepa Suslin tree}, which means that 
forcing with the Suslin tree turns it into a Kurepa tree. 
Namely, applying 
the automorphisms to a generic branch produces $\ka$-many distinct cofinal branches.

Another application of Jech's forcing for adding a Suslin tree 
is due to Stewart \cite{stewart}, who designed a countably closed 
forcing which adds a Kurepa tree. 
This fact suggests the question of whether there is a variation of 
Tennenbaum's c.c.c.\ forcing for adding a Suslin tree which adds a Kurepa tree. 
Jensen and Schlechta \cite{jensenschlechta} proved that this is not always possible: 
after forcing with the L\'{e}vy collapse to turn a Mahlo cardinal into $\omega_2$, 
there does not exist a c.c.c.\ forcing which adds a Kurepa tree. 
On the other hand, Jensen proved that if $\Box_{\omega_1}$ holds then there exists 
a c.c.c.\ forcing which adds a Kurepa tree. 
Later, Veli\v{c}kovi\'{c} \cite{boban} proved the same result with a simpler argument 
using the function $\rho$ of Todor\v{c}evi\'{c} \cite{todorpartition}, 
whose existence follows from $\Box_{\omega_1}$. 
Other examples of using $\rho$ to define c.c.c.\ forcings were given later by 
Todor\v{c}evi\'{c} \cite[Chapter 7]{todorbook}.

In light of these forcing constructions, 
a natural question is whether it is consistent that 
there exists a c.c.c.\ forcing which adds an almost Kurepa Suslin tree. 
Note that the Jensen-Schlechta limitation mentioned above 
also applies to this problem since an almost Kurepa Suslin tree is a c.c.c.\ forcing 
which adds a Kurepa tree. 
The main result of this article is that if $\Box_{\omega_1}$ holds then there exists 
a Knaster forcing with finite conditions which adds a Suslin tree with 
$\omega_2$-many almost disjoint automorphisms. 
In fact, the family of automorphisms satisfies a very strong form of almost disjointness 
which we introduce next.

Recall that if $U$ and $W$ are uncountable downwards closed subtrees of an 
$\omega_1$-tree $T$, then $U$ and $W$ are \emph{almost disjoint} if $U \cap W$ is countable. 
Observe that if $b$ and $c$ are distinct cofinal branches of $T$, then $b$ and $c$ 
are uncountable downwards closed subtrees of $T$ which are almost disjoint. 
In fact, $b \cap c$ is a countable chain, and assuming that 
$T$ has a root and is Hausdorff, 
$b \cap c$ is equal to the chain of elements less than or equal to the meet of $b$ and $c$. 
This example of cofinal branches suggests stronger forms of almost disjointness 
for uncountable downwards closed subtrees $U$ and $W$ of $T$. 
We could ask for $U \cap W$ to be a union of finitely many countable chains, 
or the slightly stronger property 
that $U \cap W$ is contained in the downward closure of a finite subset of $T$. 
Let us say that $U$ and $W$ are \emph{strongly almost disjoint} 
(or \emph{Dilworth almost disjoint}; see \cite{dilworth}) if $U \cap W$ 
is a finite union of countable chains. 
Note that an automorphism of $T$ is an uncountable 
downwards closed subtree of the tree product 
$T \otimes T$, 
so we can talk about strongly almost disjoint automorphisms of $T$ considered as 
subtrees of $T \otimes T$.

K\"{o}nig, Larson, Moore, and Veli\v{c}kovi\'{c} \cite{moorebounding} introduced 
the idea of a \emph{saturated Aronszajn tree}, 
which is an Aronszajn tree satisfying that 
any family of almost disjoint uncountable downwards closed subtrees of it 
has size at most $\omega_1$. 
A standard example of a non-saturated Aronszajn tree, due to Todor\v{c}evi\'{c} 
(see \cite[Section 2]{baumgartnerbase}), is the tree product $T \otimes K$, where 
$T$ is any Aronszajn tree and $K$ is a Kurepa tree. 
Namely, if $\{ b_\alpha : \alpha < \omega_2 \}$ is a family of cofinal branches of $K$, 
then letting $U_\alpha = T \otimes b_\alpha$ for each $\alpha < \omega_2$, the family 
$\{ U_\alpha : \alpha < \omega_2 \}$ is a witness that 
$T \otimes K$ is non-saturated.

With the above strengthening of almost disjointness at hand, we introduce the idea  
of an $\omega_1$-tree $T$ being \emph{strongly non-saturated} (or \emph{Dilworth non-saturated}), 
by which we mean that 
there exists a family of size at least $\omega_2$ consisting of 
strongly almost disjoint uncountable downwards closed subtrees of $T$. 
Since a Kurepa tree is strongly non-saturated, as witnessed by the family of its 
cofinal branches, a strongly non-saturated $\omega_1$-tree is a generalization of a Kurepa tree. 
It is not hard to show that 
if there exists a strongly non-saturated Aronszajn tree, then \textsf{CH} fails. 
The non-saturated Aronszajn tree from the previous paragraph is not strongly non-saturated. 
Namely, since $T$ is Aronszajn, we can 
fix a countable ordinal $\gamma$ such that $T_\gamma$ is infinite. 
Find $\alpha$ and $\beta$ such that the first ordinal $\delta$ for which 
$b_\alpha$ and $b_\beta$ are different on level $\delta$ is greater than $\gamma$. 
Let $x$ be the element of $b_\alpha$ and $b_\beta$ with height $\gamma$. 
Then $U_\alpha \cap U_\beta$ contains the infinite antichain $T_\gamma \times \{ x \}$, 
and therefore is not a finite union of countable chains. 
Since none of the known examples of non-saturated Aronszajn trees are strongly 
non-saturated, a natural question is whether it is consistent that there exists 
a strongly non-saturated Aronszajn tree.

\begin{theorem*}
	Assuming $\Box_{\omega_1}$, there exists a Knaster forcing $\p$ which adds a 
	Suslin tree together with a family of $\omega_2$-many strongly almost disjoint 
	automorphisms of it. 
	So $\p$ forces the existence of an almost Kurepa Suslin tree $T$ such that 
	$T \otimes T$ is a strongly non-saturated Aronszajn tree.
\end{theorem*}

The main technique used in this article is \emph{$\rho$-separation}, which is 
a variation of the notion of separation recently introduced by 
Stejskalov\'{a} and the first author \cite{KS} for the purpose of adding automorphisms to an 
$\omega_1$-tree by forcing with countable conditions. 
By working the function $\rho$ into the definition of separation, we are able to 
adapt many of the key tools of \cite{KS} to the finite context. 

We assume that the reader has a background in $\omega_1$-trees and forcing. 
Our notation is standard; we refer the reader to \cite[Section 1]{KS} for 
basic terminology and definitions concerning trees.

\section{Standard Finite Trees}

The goal of this article 
is to define a forcing which adds a Suslin tree with some remarkable properties 
by a Knaster forcing with finite conditions. 
While the tree does not exist in the ground model, it is helpful to specify the 
levels of the tree there. 
Specifically, every member of level $\alpha < \omega_1$ of the generic tree is some 
ordinal $\gamma$ such that $\omega \cdot \alpha \le \gamma < \omega \cdot (\alpha+1)$. 
With this in mind, we define the \emph{height} of a countable ordinal $\gamma$ 
to be the unique ordinal $\alpha$ such that 
$\omega \cdot \alpha \le \gamma < \omega \cdot (\alpha+1)$, and we denote the 
height of $\gamma$ by $\h(\gamma)$.

Our forcing poset consists of conditions with two components, where the 
first component is a finite tree and the second component is a finite indexed family 
of functions defined on the tree. 
In this section we develop some basic ideas about the finite trees which 
appear in our conditions.

\begin{definition}
	A \emph{standard finite tree} is a pair $(T,<_T)$ satisfying:
	\begin{enumerate}
	\item $T$ is a finite subset of $\{ 0 \} \cup (\omega_1 \setminus \omega)$ 
	and $0 \in T$;
	\item $<_T$ is a tree ordering on $T$, meaning a strict partial-ordering such that 
	for any $x \in T$, the set $\{ y \in T : y <_T x \}$ is linearly ordered by $<_T$;
	\item if $x <_T y$ then $\h(x) < \h(y)$;
	\item for all $x \in T$ and for all $\alpha \in \{ \h(z) : z \in T \} \cap \h(x)$, 
	there exists some $y \in T$ such that $\h(y) = \alpha$ and $y <_T x$.
	\end{enumerate}
\end{definition}

Suppose that $(T,<_T)$ is a standard finite tree. 
We oftentimes abbreviate $(T,<_T)$ by just $T$. 
Define $\h[T] = \{ \h(x) : x \in T \} \setminus \{ 0 \}$. 
For all $\alpha \in \h[T] \cup \{ 0 \}$, define $T_\alpha = \{ x \in T : \h(x) = \alpha \}$. 
Note that $T_\alpha$ is an antichain of $T$ by Definition 2.1(3). 
If $x \in T_\alpha$ and $\beta \in (\h[T] \cup \{ 0 \}) \cap \alpha$, 
we write $x \res_T \beta$ (or just $x \res \beta$ is $T$ is understood from context) 
for the unique $y \in T_\beta$ such that $y <_T x$, and if $X \subseteq T_\alpha$, 
define $X \res \beta = \{ x \res \beta : x \in X \}$. 
For such a set $X$, 
we say that $X$ has \emph{unique drop-downs to $\beta$} if the map 
$x \mapsto x \res \beta$ on $X$ is injective. 
If $x <_T y$, then we say that $y$ is a \emph{successor} of $x$ in $T$, 
and if $x <_T y$ and 
$\h(y) = \min(\h[T] \setminus (\h(x) + 1 ))$, then we say that 
$y$ is an \emph{immediate successor} of $x$ in $T$. 
The set of successors of $x$ in $T$ is denoted by $\Succ_T(x)$ and the set of 
immediate successors of $x$ in $T$ is denoted by 
$\ISucc_T(x)$. 

If $T$ and $U$ are standard finite trees, 
we say that $U$ is an \emph{extension} of $T$ (or $U$ \emph{extends} $T$) if 
$T \subseteq U$ and $<_T \ \subseteq \ <_U$. 
We claim that if $U$ extends $T$, then $U$ \emph{end-extends} $T$ in the sense that 
$<_T \ = \ <_U \cap \ (T \times T)$. 
For if not, then there are distinct $x, y \in T$ such that $x \not <_T y$ 
but $x <_U y$. 
By Definition 2.1(4), we can fix $z \in T$ such that $\h_T(z) = \h_T(x)$ and $z <_T y$. 
Then in $U$, $z$ and $x$ are both below $y$ but are incomparable since they have 
the same height, which contradicts that $<_U$ is a tree ordering.

Let $T$ be a standard finite tree. 
Observe that by Definition 2.1(1), $0$ is the unique element of $T$ with height $0$, and by 
Definition 2.1(4), 
$0 \le_T x$ for all $x \in T$. 
In other words, every standard finite tree has $0$ as a root. 
For any set $Y \subseteq T$, define the \emph{downward closure} of $Y$ to be the 
set $\{ z \in T : \exists x \in Y \ z \le_T x \}$. 
For any $x, y \in T$, let $x \land_T y$ (or just $x \land y$ if $T$ is understood from context) 
denote the $<_T$-largest element $z$ of $T$ 
such that $z \le_T x$ and $z \le_T y$. 
Note that $x \land y$ exists since $T$ has a root and is finite. 
A simple fact which is useful below 
is that if $x$ and $y$ are in $T$, $\alpha \le \h(x), \h(y)$, 
and $x \res \alpha \ne y \res \alpha$, then $x \land y = (x \res \alpha) \land (y \res \alpha)$.

\begin{definition}
	If $T$ and $U$ are standard finite trees, 
	we say that $U$ is a \emph{simple extension} of $T$ if:
	\begin{itemize}
	\item $U$ is an extension of $T$;
	\item $U \setminus T \subseteq \bigcup \{ U_\alpha : \alpha \in \h[U] \setminus \h[T] \}$;
	\item if $\alpha \in \h[U] \setminus \h[T]$ is less than $\max(\h[T])$ 
	and $\beta$ is the least element of $\h[T]$ greater than $\alpha$, then 
	$T_\beta$ has unique drop-downs to $\alpha$.
	\end{itemize}
\end{definition}

Note that in the third bullet point, $T_\beta = U_\beta$ so $U_\beta$ 
has unique drop-downs to $\alpha$.

We leave the easy proofs of the next two lemmas to the reader.

\begin{lemma}
	The relation on the set of all standard finite trees of being a simple extension 
	is transitive.
\end{lemma}

\begin{lemma}
	Suppose that $T$ and $U$ are standard finite trees and 
	$U$ is a simple extension of $T$. 
	Then for all $a$ and $b$ in $T$, $a \land_T b = a \land_U b$.
\end{lemma}

\begin{lemma}
	Suppose that $T$ and $U$ are standard finite trees, $U$ is a simple extension of $T$, 
	$\alpha \in \h[U] \setminus \h[T]$ is less than $\max(\h[T])$, 
	and $\beta = \min(\h[T] \setminus (\alpha+1))$. 
	Assume that $a_0 \in U_\alpha$, $a_0^+$ is the unique 
	element of $T_\beta$ above $a_0$, 
	$a_1 \in T$, and $a_0^+$ and $a_1$ are incomparable in $T$. 
	Then $a_0 \land_U a_1 = a_0^+ \land_T a_1$.
\end{lemma}

\begin{proof}
	By Lemma 2.4, $a_0^+ \land_T a_1 = a_0^+ \land_U a_1$. 
	First, assume that $\h_T(a_1) < \alpha$. 
	Since $a_0^+$ and $a_1$ are incomparable in $T$, $a_1$ is not below $a_0$ in $U$. 
	So $a_1 \ne a_0^+ \res_T \h(a_1) = a_0 \res_U \h(a_1)$, and 
	$a_0 \land_U a_1$ and 
	$a_0^+ \land_T a_1$ are both equal to 
	$(a_0^+ \res_T \h(a_1)) \land_U a_1$. 
	Secondly, assume that $\h_T(a_1) > \alpha$. 
	By the minimality of $\beta$, $\h_T(a_1) \ge \beta$. 
	Since $a_1$ is incomparable with $a_0^+$ in $T$, 
	$a_1 \res_T \beta \ne a_0^+$. 
	As $U$ is a simple extension of $T$, 
	$a_1 \res_U \alpha = (a_1 \res_T \beta) \res_U \alpha 
	\ne a_0$. 
	So $a_0 \land_U a_1$ and $a_0^+ \land_T a_1$ 
	are both equal to $a_0 \land_U (a_1 \res_U \alpha)$. 
\end{proof}

\begin{lemma}
	Suppose that $T$ is a standard finite tree and $B \subseteq \omega_1 \setminus \{ 0 \}$ 
	is a finite set 
	such that $\h[T] \subseteq B$. 
	Then there exists a standard finite tree $U$ which is a simple extension of $T$ 
	such that $\h[U] = B$.
\end{lemma}

\begin{proof}
	We prove that if $T$ is a standard finite tree 
	and $\alpha \in \omega_1 \setminus \h[T]$ is non-zero, then there exists 
	a standard finite tree $T^+$ which is a simple 
	extension of $T$ such that $\h[T^+] = \h[T] \cup \{ \alpha \}$. 
	Once we prove this statement, the lemma follows easily by induction on the size of 
	$B \setminus A$ using Lemma 2.3. 
	If $\alpha > \max(\h[T])$, pick some $x \in T$ with height equal to 
	$\max(\h[T])$, and define $T^+$ by adding an immediate successor 
	of $x$ with height $\alpha$. 
	Suppose that $\alpha < \max(\h[T])$. 
	Fix successive elements $\beta < \delta$ of $\h[T] \cup \{ 0 \}$ such that 
	$\beta < \alpha < \delta$. 
	Fix some injective mapping $x \mapsto x^-$ from $T_\delta$ into 
	$\{ z \in \omega_1 : \h(z) = \alpha \}$. 
	Define the underlying set of $T^+$ to be equal to $T \cup \{ x^- : x \in T_\delta \}$. 
	Let $<_{T^+}$ consist of the relations in $<_T$, together with the new relations  
	$x^- <_{T^+} z$ whenever $x \le_T z$ and $y <_{T^+} x^-$ whenever $y \le_T x \res \beta$.
\end{proof}

\begin{definition}
	A standard finite tree $T$ is \emph{normal} if 
	for all $x \in T$ and for all $\alpha \in \h[T] \setminus (\h(x)+1)$, 
	there exists some $y \in T$ such that $\h(y) = \alpha$ and $x <_T y$.
\end{definition}

\begin{definition}
	A standard finite tree $T$ is \emph{Hausdorff} if for every limit ordinal 
	$\delta \in \h[T]$, letting $\beta = \max(\h[T] \cap \delta)$, 
	$T_\delta$ has unique drop-downs to $\beta$.
\end{definition}

Given a standard finite tree $T$ and an element $x$ below the top level of $T$, 
it is a simple matter to define an extension $U$ of $T$ such that 
$\h[U] = \h[T]$, $U \setminus T \subseteq \ISucc_U(x)$, and $\ISucc_U(x)$ is as 
large as you want. 
In particular, repeating this process and working our way up the levels of the tree, 
we can build a normal extension of $T$. 
The next two lemmas follow from this observation. 

\begin{lemma}
	Suppose that $T$ is a standard finite tree. 
	Then there exists a standard finite tree $U$ which extends $T$ such that 
	$\h[U] = \h[T]$ and $U$ is normal.
\end{lemma}

\begin{lemma}
	Suppose that $T$ is a standard finite tree, $\alpha \in \h[T] \cap \max(\h[T])$, 
	$X \subseteq T_\alpha$, 
	and $n$ is a positive natural number such that 
	each element of $X$ has at most $n$-many immediate successors in $T$. 
	Then there exists a standard finite tree $U$ which extends $T$ such that:
	\begin{itemize}
	\item $\h[T] = \h[U]$;
	\item $U \setminus T \subseteq \bigcup \{ \ISucc_U(x) : x \in X \}$;
	\item every element of $X$ has exactly $n$-many immediate successors.
	\end{itemize}
\end{lemma}

\section{Standard Functions}

The second component of a condition in our forcing is a finite indexed family of 
partial functions defined on its standard finite tree. 
In this section we introduce and analyze some basic ideas concerning such functions.

\begin{definition}
	Let $T$ be a standard finite tree and 
	let $f$ be a partial function from $T$ to $T$. 
	\begin{itemize}
	\item $f$ is \emph{strictly increasing} if for all $x, y \in \dom(f)$, 
	if $x <_T y$ then $f(x) <_T f(y)$;
	\item $f$ is an \emph{embedding} if $f$ is strictly increasing and injective;
	\item $f$ is \emph{level preserving} if for all $x \in \dom(f)$, 
	$\h(x) = \h(f(x))$;
	\item if $f$ is level preserving and strictly increasing, then 
	$f$ is \emph{downwards closed in $T$} if whenever $x \in \dom(f)$ and 
	$\beta \in \h[T] \cap \h(x)$, then $x \res \beta \in \dom(f)$.
	\end{itemize}
\end{definition}

\begin{definition}
	Let $T$ be a standard finite tree. 
	A partial function $f$ from $T$ to $T$ is called a \emph{standard function} on $T$ 
	if it is an embedding, level preserving, downwards closed, and 
	satisfies that for all $x \in \dom(f) \setminus \{ 0 \}$, $f(x) \ne x$. 
\end{definition}

\begin{lemma}
	Let $T$ be a standard tree and let $f$ be a standard function on $T$. 
	Then $f^{-1}$ is strictly increasing. 
	Moreover, if $c$ and $d$ are in the domain of $f$ and are incomparable in $T$, then 
	$f(c)$ and $f(d)$ are also incomparable in $T$.
\end{lemma}

\begin{proof}
	Let $c, d \in \dom(f)$ and assume that $f(c) <_T f(d)$. 
	We claim that $c <_T d$. 
	Since $f$ is level preserving, $\h(c) = \h(f(c))$ and $\h(d) = \h(f(d))$. 
	So $\h(c) < \h(d)$. 
	Suppose for a contradiction that $c \not <_T d$. 
	Then $d \res \h(c) \ne c$. 
	Since $f$ is strictly increasing, $f(d \res \h(c)) <_T f(d)$. 
	So $f(c)$ and $f(d \res \h(c))$ are both equal to $f(d) \res \h(c)$, which 
	contradicts that $f$ is injective.
	
	Now assume that $c, d \in \dom(f)$ are incomparable in $T$. 
	If $c$ and $d$ have the same height, then they are different. 
	Since $f$ is injective and level preserving, 
	$f(c)$ and $f(d)$ are distinct and have the same height, and hence 
	are incomparable in $T$. 
	Assume without loss of generality that $\h(c) < \h(d)$. 
	Since $f$ is level preserving, $\h(f(c)) < \h(f(d))$. 
	So if $f(c)$ and $f(d)$ are comparable in $T$, then $f(c) <_T f(d)$. 
	Since $f^{-1}$ is strictly increasing, $c <_T d$, which is a contradiction.
\end{proof}

\begin{lemma}
	Suppose that $T$ and $U$ are standard finite trees, $U$ is an extension of $T$, 
	$\h[U] \cap (\max(h[T])+1) = \h[T]$, and $f$ is a standard function on $T$. 
	Then $f$ is a standard function on $U$.
\end{lemma}

\begin{proof}
	Clearly, $f$ is injective, strictly increasing, level preserving, and has no fixed 
	points other than $0$ (considered as a partial function from $U$ to $U$). 
	The assumption about the height function easily implies that $f$ 
	is downwards closed in $U$.
\end{proof}

For a standard finite tree $T$, 
define $T \otimes T$ to be the set of all pairs 
$(a,b)$ such that for some $\alpha \in \h[T] \cup \{ 0 \}$, 
$a$ and $b$ are in $T_\alpha$. 
A set $Z \subseteq T \otimes T$ is \emph{downwards closed} if for all 
$(a,b) \in Z$, if $(c,d) \in T \otimes T$, $c \le_T a$, and $d \le_T b$, 
then $(c,d) \in Z$. 
Note that a partial function $f$ from $T$ to $T$ is level preserving iff 
$f \subseteq T \otimes T$.
For any partial level preserving function $f$ from $T$ to $T$, 
the \emph{downward closure} of $f$ is the set of all $(a,b) \in T \otimes T$ 
such that for some $x \in \dom(f)$, $a \le_T x$ and $b \le_T f(x)$.

\begin{lemma}
	Let $T$ be a standard finite tree and let $f \subseteq T \otimes T$ be downwards closed. 
	Then $f$ is a strictly increasing function iff 
	for all pairs $(a_0,b_0)$ and $(a_1,b_1)$ in $f$, 
	$\h(a_0 \land a_1) \le \h(b_0 \land b_1)$.
\end{lemma}

\begin{proof}
	For the reverse implication, we prove the contrapositive: 
	if $f$ is not a strictly increasing function, then 
	there exist $(a_0,b_0)$ and $(a_1,b_1)$ in $f$ 
	such that $\h(b_0 \land b_1) < \h(a_0 \land a_1)$. 
	If $f$ is not a strictly increasing function, then it is either not a function 
	or it is a function but it is not strictly increasing. 
	Suppose that $f$ is not a function. 
	Then there exist $a$, $b_0$, and $b_1$ in $T$ such that $b_0 \ne b_1$ 
	and $(a,b_0)$ and $(a,b_1)$ are both in $f$. 
	Since $b_0$ and $b_1$ have the same height as $a$, and hence the 
	same height as each other, they are incomparable in $T$. 
	So $\h(b_0 \land b_1) < \h(b_0) = \h(a) = \h(a \land a)$, and the pairs 
	$(a,b_0)$ and $(a,b_1)$ are as desired. 
	Now suppose that $f$ is a function but it is not strictly increasing. 
	Then there are $c <_T d$ such that $f(c) \not <_T f(d)$. 
	Note that $\h(f(c)) = \h(c) < \h(d) = \h(f(d))$, and therefore 
	$f(c)$ and $f(d)$ are incomparable in $T$.  
	So $c \land d = c$ and $f(c) \land f(d) <_T f(c)$. 
	Hence, the pairs $(c,f(c))$ and $(d,f(d))$ are as required.

	For the forward implication, 
	we prove that if $f$ is a strictly increasing function, then 
	for any $(a_0,b_0)$ and $(a_1,b_1)$ in $f$, 
	$\h(a_0 \land a_1) \le \h(b_0 \land b_1)$. 
	In other words, we prove that for all $a_0$ and $a_1$ in the domain of $f$, 
	$\h(a_0 \land a_1) \le \h(f(a_0) \land f(a_1))$. 
	Since $f$ is downwards closed, $a_0 \land a_1 \in \dom(f)$, 
	and as $f$ is strictly increasing, $f(a_0 \land a_1) \le_T f(a_0), f(a_1)$. 
	By the definition of meet, $f(a_0 \land a_1) \le_T f(a_0) \land f(a_1)$. 
	Hence, $\h(a_0 \land a_1) = \h(f(a_0 \land a_1)) \le \h(f(a_0) \land f(a_1))$.
\end{proof}

\begin{lemma}
	Let $T$ be a standard finite tree and let $f$ be a strictly increasing, level preserving, 
	and downwards closed partial function from $T$ to $T$.
	Then $f$ is injective iff for all $a_0$ and $a_1$ in the domain of $f$, 
	$\h(a_0 \land a_1) = \h(f(a_0) \land f(a_1))$.
\end{lemma}

\begin{proof}
	By Lemma 3.5, for all $a_0$ and $a_1$ in the domain of $f$, 
	$\h(a_0 \land a_1) \le \h(f(a_0) \land f(a_1))$. 
	For the reverse implication, we prove the contrapositive: 
	suppose that $f$ is not injective, and we find $a_0$ and $a_1$ in the domain 
	of $f$ such that $\h(a_0 \land a_1) < \h(f(a_0) \land f(a_1))$. 
	Since $f$ is not injective, there are distinct $a_0$ and $a_1$ and some $b$ 
	such that $f(a_0) = b$ and $f(a_1) = b$. 
	Since $a_0$ and $a_1$ have the same height as $b$, and hence the same 
	height as each other, they are incomparable in $T$. 
	So $\h(a_0 \land a_1) < \h(a_0) = \h(b) = \h(f(a_0) \land f(a_1))$.

	For the forward implication, 
	suppose that $f$ is injective and we show that for all $a_0$ and $a_1$ 
	in the domain of $f$, $\h(a_0 \land a_1) \ge \h(f(a_0) \land f(a_1))$. 
	Suppose for a contradiction that $\h(a_0 \land a_1) < \h(f(a_0) \land f(a_1))$. 
	Let $\xi = \h(f(a_0) \land f(a_1))$. 
	Then $\h(a_0 \land a_1) < \xi$. 
	Note that $\xi \le \h(f(a_0)) = \h(a_0)$ and $\xi \le \h(f(a_1)) = \h(a_1)$, 
	so $a_0 \res \xi$ and $a_1 \res \xi$ are defined. 
	As $\h(a_0 \land a_1) < \xi$, $a_0 \res \xi \ne a_1 \res \xi$. 
	Since $f$ is downwards closed, $a_0 \res \xi$ and $a_1 \res \xi$ are in the domain of $f$, 
	and as $f$ is strictly increasing and level preserving, 
	$f(a_0 \res \xi) = f(a_0) \res \xi = f(a_0) \land f(a_1) = f(a_1) \res \xi = 
	f(a_1 \res \xi)$, which contradicts that $f$ is injective.
\end{proof}

The next lemma follows immediately from the previous one.

\begin{lemma}
	Let $T$ be a standard finite tree and let $f$ be a standard function on $T$. 
	Then for all $a_0$ and $a_1$ in the domain of $f$, 
	$\h(a_0 \land a_1) = \h(f(a_0) \land f(a_1))$.
\end{lemma}

\begin{lemma}
	Let $T$ be a standard finite tree and let $f$ be a standard function on $T$. 
	Suppose that $U$ is a simple extension of $T$. 
	Let $g$ be the downward closure of $f$ in $U$. 
	Then:
	\begin{itemize}
		\item $g$ is a standard function on $U$;
		\item $g \res T = f$;
		\item if $\alpha \in \h[U] \setminus \h[T]$ is less than $\max(\h[T])$ 
		and $\beta$ is the least element 
		of $\h[T]$ greater than $\alpha$, then 
		for all $x, y \in U_\beta$, 
		$g(x) = y$ iff $g(x \res \alpha) = y \res \alpha$.
	\end{itemize}
\end{lemma}

\begin{proof}
	We leave it to the reader to verify the second and third bullet points, 
	which is straightforward. 
	We prove that $g$ is a standard function on $U$. 
	It suffices to prove this in the special case that 
	$\h[U] \setminus \h[T]$ is a singleton. 
	For then we can prove the general statement 
	by induction on the size of $\h[U] \setminus \h[T]$. 
	So let $\h[U] \setminus \h[T]$ consist of one ordinal $\alpha$. 
	If $\alpha > \max(\h[T])$, then $f = g$, and easily $f$ is a standard function on $U$. 
	Assume that $\alpha < \max(\h[T])$ and let $\beta = \min(\h[T] \setminus (\alpha+1))$. 

	Since $f$ is downwards closed, if $(a,b) \in g \setminus f$, 
	then $a$ and $b$ have height $\alpha$ and hence neither of them are in $T$. 
	So $(a,b) \in g \setminus f$ iff $a$ and $b$ have height $\alpha$ and there are 
	$a^+$ and $b^+$ in $T_\beta$ above $a$ and $b$ respectively such that 
	$f(a^+) = b^+$. 
	If $a = b$, then since $U$ is a simple extension of $T$, 
	$a^+ = b^+$, contradicting the fact that $f$ is a standard function. 
	So assuming that $g$ is a function, it has no fixed points other than $0$.

	By definition, $g \subseteq U \otimes U$ is downwards closed, so if $g$ is a function 
	then it is obviously level preserving. 
	So it suffices to prove that $g$ is an injective 
	strictly increasing function. 
	By Lemmas 3.5 and 3.6, it suffices to prove that whenever $(a_0,b_0)$ and $(a_1,b_1)$ 
	are in $g$, then $\h(a_0 \land_U a_1) = \h(b_0 \land_U b_1)$. 
	We consider three cases.

	Case 1: $(a_0,b_0)$ and $(a_1,b_1)$ are both in $f$. 
	By Lemmas 2.4 and 3.7, $\h(a_0 \land_U a_1) = \h(a_0 \land_T a_1) = 
	\h(b_0 \land_T b_1) = \h(b_0 \land_U b_1)$.

	Case 2: $(a_0,b_0)$ and $(a_1,b_1)$ are both in $g \setminus f$. 
	Then $a_0$, $b_0$, $a_1$, and $b_1$ are in $U_\alpha$. 
	Fix $a_0^+$, $b_0^+$, $a_1^+$, and $b_1^+$ 
	in $T_\beta$ above $a_0$, $b_0$, $a_1$, and $b_1$ respectively 
	such that $f(a_0^+) = b_0^+$ and $f(a_1^+) = b_1^+$. 
	First, assume that $a_0 \ne a_1$. 
	Then $a_0^+ \ne a_1^+$, and since $f$ is injective, $b_0^+ \ne b_1^+$. 
	By unique drop-downs, $b_0 \ne b_1$. 
	Hence, $a_0 \land_U a_1 = a_0^+ \land_T a_1^+$ and 
	$b_0 \land_U b_1 = b_0^+ \land_T b_1^+$. 
	So $\h(a_0 \land_U a_1) = \h(a_0^+ \land_T a_1^+)$, which by Lemma 3.7 is equal to 
	$\h(b_0^+ \land_T b_1^+) = \h(b_0 \land_U b_1)$.
	Secondly, assume that $a_0 = a_1$. 
	By unique drop-downs, $a_0^+ = a_1^+$. 
	So $b_0^+ = b_1^+$, and hence $b_0 = b_1$. 
	Therefore, $\h(a_0 \land a_1) = \alpha = \h(b_0 \land b_1)$.

	Case 3: One of $(a_0,b_0)$ and $(a_1,b_1)$ is in 
	$g \setminus f$ and the other is in $f$. 
	Without loss of generality, assume that $(a_0,b_0)$ is in $g \setminus f$ and $(a_1,b_1)$ 
	is in $f$. 
	Then $a_0$ and $b_0$ are in $U_\alpha$, and there are $a_0^+$ and $b_0^+$ 
	above $a_0$ and $b_0$ respectively with height $\beta$ such that 
	$f(a_0^+) = b_0^+$. 
	If $a_1 \ge_T a_0^+$, then since $f$ is strictly increasing, $b_1 \ge_T b_0^+$. 
	Hence, $a_0 <_U a_1$ and $b_0 <_U b_1$. 
	So $\h(a_0 \land_U a_1) = \h(a_0) = \h(b_0) = \h(b_0 \land_U b_1)$. 
	If $a_1 <_T a_0^+$, then since $f$ is strictly increasing, $b_1 <_T b_0^+$, 
	so $a_1 <_U a_0$ and $b_1 <_U b_0$. 	
	Therefore, $\h(a_0 \land_U a_1) = \h(a_1) = \h(b_1) = \h(b_0 \land_U b_1)$.
	
	Finally, assume that $a_0^+$ and $a_1$ are incomparable in $T$. 
	By Lemma 3.3, $b_0^+$ and $b_1$ are also incomparable in $T$. 
	Since $U$ is a simple extension of $T$, by Lemma 2.5 we have that 
	$a_0 \land_U a_1 = a_0^+ \land_T a_1$ and 
	$b_0 \land_U b_1 = b_0^+ \land_T b_1$. 
	By Lemma 3.7, $\h(a_0^+ \land_T a_1) = \h(b_0^+ \land_T b_1)$. 
	So $\h(a_0 \land_U a_1) = \h(a_0^+ \land_T a_1) = 
	\h(b_0^+ \land_T b_1) = \h(b_0 \land_U b_1)$.
\end{proof}

\section{Consistency and \texorpdfstring{$\rho$}{rho}-Separation}

In this section, we introduce and develop some of the main tools we use to define and analyze 
our forcing, namely, 
consistency and $\rho$-separation. 
These ideas are natural modifications to the finite context 
of the notions of consistency and separation 
which were introduced recently by Stejskalov\'{a} and the first author for the purpose 
of forcing automorphisms of an $\omega_1$-tree with countable conditions. 
Roughly speaking, $\rho$-separation allows configurations of indexed families of automorphisms 
which are prohibited by the original definition of separation, 
but only if they occur low enough in the tree according to 
the function $\rho$.

\begin{definition}[Consistency]
	Let $T$ be a standard finite tree and let $f$ be a standard function on $T$. 
	Let $\beta < \alpha$ be ordinals in $\h[T]$. 
	Suppose that $X \subseteq T_\alpha$ and $X$ has unique drop-downs to $\beta$. 
	We say that $X \res \beta$ and $X$ are \emph{$f$-consistent} if for all 
	$x, y \in X$, $f(x \res \beta) = y \res \beta$ iff $f(x) = y$.
\end{definition}

Note that in the above definition, since $f$ is strictly increasing, 
$f(x) = y$ implies that $f(x \res \beta) = y \res \beta$. 
So consistency is equivalent to the upward direction of the definition, namely, 
that for all $x, y \in X$, 
if $f(x \res \beta) = y \res \beta$ then $f(x) = y$.

The proof of the first part of the following lemma is routine, and the 
second part follows immediately from Lemma 3.8.

\begin{lemma}
	Let $T$ be a standard finite tree and let $f$ be a standard function on $T$.
	\begin{enumerate}
	\item Let $\beta < \delta < \alpha$ be ordinals in $\h[T]$. 
	Suppose that $X \subseteq T_\alpha$, $X$ has unique drop-downs to $\beta$, 
	and $X \res \beta$ and $X$ are $f$-consistent. 
	Then $X \res \delta$ has unique drop-downs to $\beta$ and 
	$X \res \beta$ and $X \res \delta$ are $f$-consistent.
	\item Suppose that $U$ is a simple extension of $T$ and $g$ 
	is the downward closure of $f$ in $U$. 
	If $\alpha \in \h[U] \setminus \h[T]$ is less than $\max(\h[T])$ 
	and $\beta$ is the least element of $\h[T]$ 
	above $\alpha$, then $U_\beta \res \alpha$ and $U_\beta$ are $g$-consistent.	
	\end{enumerate}
\end{lemma}

\begin{definition}
	Let $T$ be a standard finite tree and let $\{ f_\xi : \xi \in A \}$ be a finite 
	indexed family of standard functions on $T$. 
	Let $\alpha \in \h[T]$.
	\begin{itemize}
	\item (Separation) 
	Let $a_0,\ldots,a_{n-1}$ be distinct elements of $T_\alpha$. 
	We say that $\{ f_\xi : \xi \in A \}$ is \emph{separated on $(a_0,\ldots,a_{n-1})$} 
	if for all $i < n$, there exists at most one triple $(j,m,\tau)$ 
	such that $j < i$, $m \in \{ 1, -1 \}$, $\tau \in A$, 
	and $f_{\tau}^m(a_i) = a_j$.
	\item (Separation for Sets) 	
	Let $X \subseteq T_\alpha$. 
	We say that $\{ f_\xi : \xi \in A \}$ is \emph{separated on $X$} if there 
	exists some injective tuple $\vec a = (a_0,\ldots,a_{n-1})$ which lists the elements of 
	$X$ such that $\{ f_\xi : \xi \in A \}$ is separated on $\vec a$.
	\end{itemize}
\end{definition}

Under the assumptions of the above definition, 
for $i, j < n$ we sometimes refer to an equation of the 
form $f_\tau^m(x) = y$, where $\tau \in A$ and $m \in \{ -1, 1 \}$, 
as a \emph{relation between $x$ and $y$ 
with respect to $\{ f_\xi : \xi \in A \}$}. 
Note that if $\{ f_\xi : \xi \in A \}$ is separated on $(a_0,\ldots,a_{n-1})$, then 
for any $B \subseteq A$, $\{ f_\xi : \xi \in B \}$ is 
separated on $(a_0,\ldots,a_{n-1})$.

\begin{lemma}
	Let $T$ be a standard finite tree and let $\{ f_\xi : \xi \in A \}$ be a finite 
	indexed family of standard functions on $T$. 
	Suppose that $\alpha \in \h[T]$, 
	$X \subseteq T_\alpha$, and $\{ f_\xi : \xi \in A \}$ is separated on $X$. 
	Then for all $x, y \in X$ and for all pairs $(m,\tau)$ and $(n,\sigma)$ in 
	$\{ -1, 1 \} \times A$, if $f_\tau^m(x) = y$ and $f_\sigma^n(x) = y$, then 
	$(m,\tau) = (n,\sigma)$.
\end{lemma}

\begin{proof}
	Fix an injective tuple $\vec a = (a_0,\ldots,a_{q-1})$ which lists the elements of $X$ 
	so that $\{ f_\xi : \xi \in A \}$ is separated on $\vec a$. 
	Fix $i, j < q$ so that $a_i = x$ and $a_j = y$. 
	By replacing $m$ and $n$ with $-m$ and $-n$ if necessary, 
	we may assume without loss of generality that $j < i$. 
	Then $f_\tau^m(a_j) = a_i$ and $f_\sigma^n(a_j) = a_i$, which by separation 
	imply that the triples $(i,m,\tau)$ and $(i,n,\sigma)$ are equal. 
	So $(m,\tau) = (n,\sigma)$.
\end{proof}

\begin{lemma}[Strong Persistence]
	Let $T$ be a standard finite tree and let $\{ f_\xi : \xi \in A \}$ be a finite 
	indexed family of standard functions on $T$. 
	Let $\alpha < \beta$ be in $\h[T]$ and let $X \subseteq T_\alpha$. 
	Assume that $\{ f_\xi : \xi \in A \}$ is separated on $X$. 
	Then $\{ f_\xi : \xi \in A \}$ is separated on the set 
	$Y = \{ b \in T_\beta : \exists x \in X \ x <_T b \}$.
\end{lemma}

\begin{proof}
	Fix an injective tuple $\vec a = (a_0,\ldots,a_{n-1})$ which lists the elements 
	of $X$ so that $\{ f_\xi : \xi \in A \}$ is separated on $\vec a$. 
	Now let $\vec b = (b_0,\ldots,b_{p-1})$ be any injective tuple which lists 
	the elements of $Y$ so that for all $j < i < p$, 
	if $j^*$ and $i^*$ are such that 
	$b_j \res \alpha = a_{j^*}$ and $b_i \res \alpha = a_{i^*}$, 
	then $j^* \le i^*$. 
	For each $i < p$, let $i^* < n$ be such that $b_i \res \alpha = a_{i^*}$.
	
	Suppose for a contradiction that $i < p$ and 
	there exist distinct triples $(j_0,m_0,\tau_0)$ and $(j_1,m_1,\tau_1)$ such that 
	for each $k < 2$, $j_k < i$, $m_k \in \{ 1, -1 \}$, $\tau_k \in A$, 
	and $f_{\tau_k}^{m_k}(b_i) = b_{j_k}$. 
	Consider $k < 2$. 
	Then $j_k^* \le i^*$. 
	Since $f_{\tau_k}^{m_k}$ is strictly increasing, 
	$f_{\tau_k}^{m_k}(a_{i^*}) = a_{j_k^*}$. 
	As $f_{\tau_k}$ has no fixed points other than $0$, $j_{k^*} < i^*$. 
	Because $\{ f_\xi : \xi \in A \}$ is separated on $\vec a$, the triples 
	$(j_0^*,m_0,\tau_0)$ and $(j_1^*,m_1,\tau_1)$ are equal. 
	Hence, $m_0 = m_1$ and $\tau_0 = \tau_1$. 
	So $b_{j_0} = f_{\tau_0}^{m_0}(b_i) = f_{\tau_1}^{m_1}(b_i) = b_{j_1}$. 
	Therefore, $j_0 = j_1$. 
	But then the triples 
	$(j_0,m_0,\tau_0)$ and $(j_1,m_1,\tau_1)$ are equal, 
	which is a contradiction.
\end{proof}

For the remainder of the article we assume that $\Box_{\omega_1}$ holds, and we work with a 
function $\rho$ whose existence follows from $\Box_{\omega_1}$. 
The function $\rho$ was introduced by Todor\v{c}evi\'{c} \cite[Section 2]{todorpartition}. 
The basic properties of $\rho$ which we use are as follows:
\begin{itemize}
	\item $\rho$ is a function with domain $\omega_2^2$ and codomain $\omega_1$;
	\item $\rho(\alpha,\alpha) = 0$ for all $\alpha < \omega_2$;
	\item $\rho(\alpha,\beta) = \rho(\beta,\alpha)$ for all $\alpha, \beta < \omega_2$;
	\item let $F$ be an uncountable family of finite subsets of $\omega_2$ and 
	let $\mu < \omega_1$; then there exists an uncountable set $F' \subseteq F$ 
	such that for all distinct $x$ and $y$ in $F'$, 
	for all $\tau \in x \setminus y$, $\zeta \in y \setminus x$, 
	and $\gamma \in x \cap y$, 
	$$
	\rho(\tau,\zeta) \ge \max\{ \min \{ \rho(\tau,\gamma), \rho(\zeta,\gamma) \}, \mu \}.
	$$
	\end{itemize}

We refer to the last bullet point above as the \emph{special property of $\rho$}. 
It was proven in \cite[Lemma 7.4.7]{todorbook} (also see \cite[Lemma 4.5]{boban}).

For the remainder of this section, whenever we mention a finite indexed family, 
we assume implicitly that the index set is a subset of $\omega_2$.

We now introduce our $\rho$-variation of the concept of separation.

\begin{definition}
	Let $T$ be a standard finite tree and let $\{ f_\xi : \xi \in A \}$ be a finite 
	indexed family of standard functions on $T$. 
	Let $\alpha \in \h[T]$.
	\begin{itemize}
	\item ($\rho$-Separation) 
	Let $a_0,\ldots,a_{n-1}$ be distinct elements of $T_\alpha$. 
	We say that $\{ f_\xi : \xi \in A \}$ is 
	\emph{$\rho$-separated on $(a_0,\ldots,a_{n-1})$} if 
	for all $i < n$, if $(j_0,m_0,\tau_0)$ and $(j_1,m_1,\tau_1)$ are distinct triples 
	satisfying that for each $k < 2$, $j_k < i$, $m_k \in \{ -1, 1 \}$, 
	$\tau_k \in A$, and $f_{\tau_k}^{m_k}(a_i) = a_{j_k}$, then 
	$j_0 = j_1$ and $\rho(\tau_0,\tau_1) \ge \alpha$.
	\item ($\rho$-Separation for Sets) Let $X \subseteq T_\alpha$. 
	We say that $\{ f_\xi : \xi \in A \}$ is \emph{$\rho$-separated on $X$} if there 
	exists some injective tuple $\vec a = (a_0,\ldots,a_{n-1})$ which lists the elements of 
	$X$ such that $\{ f_\xi : \xi \in A \}$ is $\rho$-separated on $\vec a$.
	\end{itemize}
\end{definition}

Note that separation implies $\rho$-separation. 
If $\{ f_\xi : \xi \in A \}$ is $\rho$-separated on the tuple $(a_0,\ldots,a_{n-1})$, then 
for any $B \subseteq A$, $\{ f_\xi : \xi \in B \}$ is 
$\rho$-separated on $(a_0,\ldots,a_{n-1})$.

\begin{lemma}
	Let $T$ be a standard finite tree and let $\{ f_\xi : \xi \in A \}$ be a finite 
	indexed family of standard functions on $T$. 
	Let $\alpha \in \h[T]$. 
	\begin{enumerate}
	\item Let $a_0,\ldots,a_{n-1}$ be distinct elements of $T_\alpha$. 
	Suppose that for all $\sigma$ and $\tau$ in $A$, $\rho(\sigma,\tau) < \alpha$. 
	Then $\{ f_\xi : \xi \in A \}$ is separated on $(a_0,\ldots,a_{n-1})$ iff 
	$\{ f_\xi : \xi \in A \}$ is $\rho$-separated on $(a_0,\ldots,a_{n-1})$.
	\item Let $X \subseteq T_\alpha$. 
	Suppose that $\{ f_\xi : \xi \in A \}$ is $\rho$-separated on $X$. 
	If $Y \subseteq T_\alpha \setminus X$ and for all $\tau \in A$, 
	$Y$ is disjoint from the domain and range of $f_\tau$, 
	then $\{ f_\xi : \xi \in A \}$ is $\rho$-separated on $X \cup Y$.
	\end{enumerate}
\end{lemma}

\begin{proof}
	(1) is easy. 
	(2) Fix an injective tuple $\vec a$ which lists the elements 
	of $X$ so that $\{ f_\xi : \xi \in A \}$ 
	is $\rho$-separated on $\vec a$. 
	Let $\vec b$ be any injective tuple which lists the elements of $Y$. 
	Since there are no relations between the elements of $Y$ and $X \cup Y$ 
	with respect to $\{ f_\xi : \xi \in A \}$, 
	it easily follows that 
	$\{ f_\xi : \xi \in A \}$ is $\rho$-separated on the concatenation 
	$\vec a^{\frown}\vec b$.
\end{proof}

\begin{lemma}[Downward Persistence]
	Let $T$ be a standard finite tree and let $\{ f_\xi : \xi \in A \}$ be a finite 
	indexed family of standard functions on $T$. 
	Let $\alpha < \beta$ be in $\h[T]$. 
	Suppose that $X \subseteq T_\beta$ has unique drop-downs 
	to $\alpha$ and for all $\tau \in A$, 
	$X \res \alpha$ and $X$ are $f_\tau$-consistent. 
	If $\{ f_\xi : \xi \in A \}$ is $\rho$-separated on $X$, 
	then $\{ f_\xi : \xi \in A \}$ is $\rho$-separated on $X \res \alpha$.
\end{lemma}

\begin{proof}
	Let $\vec a = (a_0,\ldots,a_{n-1})$ be an injective tuple which lists $X$ 
	so that $\{ f_\xi : \xi \in A \}$ is $\rho$-separated $\vec a$. 
	By unique drop-downs, 
	$\vec a \res \alpha = (a_0 \res \alpha,\ldots,a_{n-1} \res \alpha)$ 
	is an injective tuple which lists the elements of $X \res \alpha$. 
	We claim that $\{ f_\xi : \xi \in A \}$ is $\rho$-separated on $\vec a \res \alpha$. 
	Suppose that $i < n$ and $(j_0,m_0,\tau_0)$ and $(j_1,m_1,\tau_1)$ 
	are distinct triples satisfying that for each $k < 2$, 
	$j_k < i$, $m_k \in \{ -1, 1 \}$, $\tau_k \in A$, 
	and $f_{\tau_k}^{m_k}(a_i \res \alpha) = a_{j_k} \res \alpha$. 
	By consistency, $f_{\tau_k}^{m_k}(a_i) = a_{j_k}$. 
	As $\{ f_\xi : \xi \in A \}$ is $\rho$-separated on $\vec a$, 
	$j_0 = j_1$ and $\rho(\tau_0,\tau_1) \ge \beta > \alpha$.
\end{proof}

\begin{lemma}
	Let $T$ be a standard finite tree and let 
	$\{ f_\xi : \xi \in A \}$ be a finite set of standard functions on $T$. 
	Suppose that $U$ is a simple extension of $T$. 
	For each $\tau \in A$, let $g_\tau$ be the downward closure of $f_\tau$ in $U$. 
	Assume that $\alpha \in \h[U] \setminus \h[T]$ is less than $\max(\h[T])$, 
	$\beta$ is the least element of $\h[T]$ above $\alpha$, and 
	$\{ f_\xi : \xi \in A \}$ is $\rho$-separated on $T_\beta$. 
	Then $\{ g_\tau : \tau \in A \}$ is $\rho$-separated on $U_\alpha$. 
	In particular, if $\{ f_\xi : \xi \in A \}$ is $\rho$-separated 
	on $T_\beta$ for all $\beta \in \h[T]$, then $\{ g_\tau : \tau \in A \}$ 
	is $\rho$-separated for all $\beta \in \h[U]$.
\end{lemma}

\begin{proof}
	Immediate from Lemmas 4.2(2), 4.7(2), and 4.8 (Downward Persistence).
\end{proof}

\begin{proposition}[Characterization of $\rho$-Separation]
	Let $T$ be a standard finite tree and let $\{ f_\xi : \xi \in A \}$ be a finite 
	indexed family of standard functions on $T$. 
	Let $\alpha \in \h[T]$ and let $X \subseteq T_\alpha$. 
	Then $\{ f_\xi : \xi \in A \}$ is $\rho$-separated on $X$ if and only if:
	\begin{enumerate}
		\item for all $x, y \in X$, if $(m_0,\tau_0)$ and $(m_1,\tau_1)$ 
		are distinct pairs, where for each $k < 2$, 
		$m_k \in \{ -1, 1 \}$, 
		$\tau_k \in A$, and $f_{\tau_k}^{m_k}(x) = y$, 
		then $\rho(\tau_0,\tau_1) \ge \alpha$;
		\item there does not exist a \emph{loop with respect to $\{ f_\xi : \xi \in A \}$}, 
		by which we mean a sequence 
		$\langle c_0,\ldots,c_{p-1} \rangle$ such that 
		$p \ge 4$, $\langle c_0,\ldots,c_{p-2} \rangle$ is injective, 
		$c_0 = c_{p-1}$, and for all $i < p-1$ there exists some 
		$\tau \in A$ and $m \in \{ -1, 1 \}$ such that $f_\tau^m(c_i) = c_{i+1}$.
	\end{enumerate}
\end{proposition}

\begin{proof}
	Suppose that $\{ f_\xi : \xi \in A \}$ is $\rho$-separated on $X$, and let 
	$\vec a = (a_0,\ldots,a_{n-1})$ be an injective tuple which lists $X$ 
	such that $\{ f_\xi : \xi \in A \}$ is separated on $\vec a$. 
	We prove that (1) and (2) hold.

	(1) Suppose that $x, y \in X$ 
	and $(m_0,\tau_0)$ and $(m_1,\tau_1)$ are distinct pairs 
	satisfying that for each $k < 2$, 
	$m_k \in \{ -1, 1 \}$, 
	$\tau_k \in A$, and $f_{\tau_k}^{m_k}(x) = y$. 
	We prove that $\rho(\tau_0,\tau_1) \ge \alpha$. 
	Fix $j$ and $i$ less than $n$ such that $x = a_i$ and $y = a_j$. 
	By replacing $m_0$ and $m_1$ by $-m_0$ and $-m_1$ if necessary, 
	we may assume without loss of generality that $j < i$. 
	Since the triples $(j,m_0,\tau_0)$ and $(j,m_1,\tau_1)$ are distinct, 
	it follows by $\rho$-separation that $\rho(\tau_0,\tau_1) \ge \alpha$.

	(2) Suppose for a contradiction that there exists a sequence 
	$\langle c_0,\ldots,c_{p-1} \rangle$ such that 
	$p \ge 4$, $\langle c_0,\ldots,c_{p-2} \rangle$ is injective, 
	$c_0 = c_{p-1}$, and for all $i < p-1$ there exists some 
	$\tau_i \in A$ and $m_i \in \{ -1, 1 \}$ such that $f_{\tau_i}^{m_i}(c_i) = c_{i+1}$. 
	For each $i < p$ fix $j_i < n$ such that $c_i = a_{j_i}$. 
	By shifting the sequence if necessary, we may assume without loss of generality 
	that $j_{0} = \max \{ j_i : i < p-1 \}$. 
	Let $j^* = j_0$.  
	Note that $j^* = j_{p-1}$, $j_1 < j^*$, and $j_{p-2} < j^*$.  
	Then $f_{\tau_{0}}^{m_0}(a_{j^*}) = a_{j_1}$ and 
	$f_{\tau_{p-2}}^{-m_{p-2}}(a_{j^*}) = a_{j_{p-2}}$. 
	But the triples $(j_1,m_0,\tau_0)$ and $(j_{p-2},-m_{p-2},\tau_{p-2})$ 
	are distinct and $j_1 \ne j_{p-2}$, which contradicts $\rho$-separation.

	Conversely, assume that (1) and (2) hold and we prove that 
	$\{ f_\xi : \xi \in A \}$ is $\rho$-separated on $X$. 
	We build by induction an injective tuple 
	$\vec a = (a_0,\ldots,a_{n-1})$ which lists the elements of $X$ so that 
	$\{ f_\xi : \xi \in A \}$ is $\rho$-separated on $\vec a$. 
	This tuple splits into consecutive segments satisfying:
	\begin{enumerate}
	\item[(a)] if $c$ and $d$ are in $X$ and are in distinct segments of $\vec a$, 
	then there does not exist $(m,\tau) \in \{ -1, 1 \} \times A$ 
	such that $f_\tau^m(c) = d$;
	\item[(b)] for each member $b$ of a segment different from 
	the first element $a$ of that segment, 
	there exists a finite sequence $\langle c_0,\ldots,c_{p} \rangle$ such that 
	$c_{0} = b$, $c_{p} = a$, 
	and for all $i < p$, $c_{i+1}$ appears earlier in the segment than 
	$c_{i}$ does and for some $(m, \tau) \in \{ -1, 1 \} \times A$, 
	$f_\tau^m(c_i) = c_{i+1}$. 
	\end{enumerate}

	For the first member of the first segment, let $a_0$ be an arbitrary member of $X$. 
	Assuming that $(a_0,\ldots,a_k)$ has been defined and is part of the first segment, 
	let $a_{k+1}$ be any element of $X \setminus \{ a_0,\ldots,a_k \}$ satisfying that 
	there exist $i \le k$, $\tau \in A$, and $m \in \{-1,1\}$ 
	such that $f_\tau^m(a_{k+1}) = a_i$. 
	Note that (b) holds for $a_{k+1}$ assuming that it holds for $a_i$. 
	If there does not exist such an element $a_{k+1}$, 
	then we move on to the next segment using the same instructions as above, by 
	picking the first element of the next segment arbitrarily and successively choosing 
	new elements related to earlier members of the segment. 
	We continue in this manner defining a sequence of segments, 
	and stop once we have listed all of the members of $X$. 
	This completes the definition of $\vec a = (a_0,\ldots,a_{n-1})$, and 
	properties (a) and (b) clearly hold.

	Let us prove that $\{ f_\xi : \xi \in A \}$ is $\rho$-separated on $\vec a$. 
	Consider $i < n$, and suppose that 
	$(j_0,m_0,\tau_0)$ and $(j_1,m_1,\tau_1)$ are distinct triples 
	satisfying that for each $k < 2$, $j_k < i$, $m_k \in \{ -1, 1 \}$, 
	$\tau_k \in A$, and $f_{\tau_k}^{m_k}(a_i) = a_{j_k}$. 
	We prove that $j_0 = j_1$ and $\rho(\tau_0,\tau_1) \ge \alpha$.

	By (a), $a_i$, $a_{j_0}$, and $a_{j_1}$ must belong to the same segment of $\vec a$. 
	Let $d$ be the first member of this segment. 
	We claim that $j_0 = j_1$. 
	Suppose not. 
	For each $k < 2$, let $\vec c^k = (c^k_0,\ldots,c^k_{p^k})$ 
	be a sequence as described in (b) satisfying that 
	$c^k_0 = a_{j_k}$ and $c^k_{p^k} = d$. 
	Since these sequences both end at $d$, we can fix 
	fix $1 \le q^0 \le p^0$ and $1 \le q^1 \le p^1$ such that 
	$c^0_{q^0} = c^1_{q^1}$ but for all $r < q^0$ and $s < q^1$, $c^0_r \ne c^1_s$. 
	Now the sequence resulting from the concatenation of the four sequences 
	$\langle a_i \rangle$, 
	$\langle c^0_0,\ldots,c^0_{q^0} \rangle$, the reverse of 
	$\langle c^1_0,\ldots,c^1_{q^1-1} \rangle$, and $\langle a_i \rangle$ 
	is a loop, contradicting (2).

	So indeed $j_0 = j_1$. 
	Let $j^* = j_0$. 
	Then the pairs $(m_0,\tau_0)$ and $(m_1,\tau_1)$ are distinct 
	and for each $k < 2$, $f_{\tau_k}^{m_k}(a_i) = a_{j}$. 
	By (1), $\rho(\tau_0,\tau_1) \ge \alpha$.
	\end{proof}

The proof of the following is very similar to the proof of 
\cite[Proposition 5.12]{KS}.

\begin{proposition}[$1$-Key Property]
	Let $T$ be a standard finite tree which is normal 
	and let $\{ f_\xi : \xi \in A \}$ be a finite 
	indexed family of standard functions on $T$. 
	Assume:
	\begin{itemize}
	\item $\alpha < \beta$ are in $\h[T]$; 
	\item $X \subseteq T_\alpha$;
	\item $\{ f_\xi : \xi \in A \}$ is separated on $X$;
	\item whenever $f_\tau(x) = y$ holds, where 
	$x, y \in X$ and $\tau \in A$, 
	then $\Succ_T(x) \subseteq \dom(f_\tau)$ and $\Succ_T(y) \subseteq \ran(f_\tau)$.
	\end{itemize}
	Then for all $b \in T_\beta$ such that $b \res \alpha \in X$, 
	there exists a set $Y \subseteq T_\beta$ with unique drop-downs to $\alpha$ 
	such that $Y \res \alpha = X$, $b \in Y$, and for all $\tau \in A$, 
	$X$ and $Y$ are $f_\tau$-consistent.
\end{proposition}

\begin{proof}
	Let $\vec a = (a_0,\ldots,a_{n-1})$ be an injective tuple which lists the elements 
	of $X$ in such a way that $\{ f_\xi : \xi \in A \}$ is separated on $\vec a$. 
	Fix $\bar{n} < n$ such that $b \res \alpha = a_{\bar n}$.
	
	We claim that there exists a sequence 
	$$
	\langle i_0,(i_1,m_1,\tau_1),\ldots,(i_{l-1},m_{l-1},\tau_{l-1}) \rangle, 
	$$
	for some $l \le \bar{n}+1$, such that:
	\begin{enumerate}
	\item $\bar{n} = i_0 > i_1 > \cdots > i_{l-1} \ge 0$;
	\item for all $0 < k < l$, $\tau_k \in A$, $m_k \in \{ -1, 1 \}$, and 
	$f_{\tau_k}^{m_k}(a_{i_{k-1}}) = a_{i_{k}}$; 
	\item there does not exist a triple $(i,m,\tau)$ such that 
	$i < i_{l-1}$, $m \in \{ -1, 1 \}$, 
	$\tau \in A$, and $f_\tau^m(a_{i_{l-1}}) = a_i$.
	\end{enumerate}

	We construct the desired sequence by induction. 
	Let $i_{0} = \bar{n}$. 
	Now let $k \ge 0$ and assume that we have defined 
	$\langle i_0,(i_1,m_1,\tau_1),\ldots,(i_{k},m_{k},\tau_{k}) \rangle$ 
	as described in (1) and (2). 
	If there does not exist a triple $(i,m,\tau)$ such that $i < i_k$, $m \in \{ -1, 1 \}$, 
	$\tau \in A$, and $f_\tau^m(a_{i_k}) = a_i$, 
	then let $l = k+1$ and we are done. 
	Otherwise, fix such a triple $(i,m,\tau)$ and let 
	$i_{k+1} = i$, $m_{k+1} = m$, and $\tau_{k+1} = \tau$. 
	This completes the construction. 
	Note that (1) implies that $l \le \bar{n}+1$.

	Define a tuple $(c_0,\ldots,c_{l-1})$ by induction as follows, maintaining that 
	for all $k < l$, $a_{i_k} <_T c_k$. 
	Let $c_0 = b$. 
	Then $a_{i_0} = a_{\bar n} <_T b = c_0$. 
	Suppose that $0 < k < l$ and $c_{k-1}$ is defined so that 
	$a_{i_{k-1}} <_T c_{k-1}$. 
	By (2), $f_{\tau_{k}}^{m_k}(a_{i_{k-1}}) = a_{i_k}$. 
	By the last bullet point in the assumptions of the proposition, 
	$c_{k-1}$ is in the domain of $f_{\tau_k}^{m_k}$. 
	Define $c_k = f_{\tau_k}^{m_k}(c_{k-1})$. 
	Since $f_{\tau_k}^{m_k}$ is strictly increasing, $a_{i_k} <_T c_k$. 
	This completes the definition of $(c_0,\ldots,c_{l-1})$.

	We now construct the set $Y$ as described in the conclusion of the proposition. 
	By induction on $i < n$ we choose $b_i$ in $T_\beta$ above $a_i$, 
	and let $Y = \{ b_0,\ldots,b_{n-1} \}$. 
	We maintain that for all $k < n$:
	\begin{enumerate}
	\item[(a)] for all $\tau \in A$, 
	$(a_0,\ldots,a_{k})$ and $(b_0,\ldots,b_{k})$ are $f_\tau$-consistent;
	\item[(b)] for all $m < l$, if $i_m \le k$ then $b_{i_m} = c_m$.
	\end{enumerate} 
	Assuming that we are able to define $(b_0,\ldots,b_{n-1})$ 
	with these properties, then for all $\tau \in A$, 
	$(a_0,\ldots,a_{n-1})$ and $(b_0,\ldots,b_{n-1})$ are $f_\tau$-consistent, 
	and $b_{\bar{n}} = b_{i_0} = c_0 = b$, which completes the proof.

	For the base case, if $0 \in \{ i_0, \ldots, i_{l-1} \}$, 
	then clearly $0 = i_{l-1}$, so in this case we let $b_0 = b_{i_{l-1}} = c_{l-1}$. 
	Otherwise, let $b_0$ be an arbitrary element of $T_\beta$ above $a_0$. 
	Clearly, the inductive hypotheses are maintained.

	Now let $0 < k < n$ and assume that we have chosen $b_i$ for all 
	$i < k$ so that the tuple $(b_0,\ldots,b_{k-1})$ satisfies (a) and (b).

	Case 1:	There does not exist a triple $(j,m,\tau)$ such that 
	$j < k$, $m \in \{ -1, 1 \}$, $\tau \in A$, and 
	$f_{\tau}^m(a_k) = a_j$. 
	If $k \in \{ i_0, \ldots, i_{l-1} \}$, then clearly $k = i_{l-1}$, 
	and we let $b_{k} = b_{i_{l-1}} = c_{l-1}$. 
	So inductive hypothesis (b) holds. 
	Otherwise, choose $b_{k}$ above $a_{k}$ arbitrarily. 
	The inductive hypothesis together with the fact that 
	there are no relations between $a_k$ and members of $(a_0,\ldots,a_{k-1})$ 
	with respect to $\{ f_\xi : \xi \in A \}$  
	easily imply inductive hypothesis (a).

	Case 2:	There exists a triple $(j,m,\sigma)$ such that 
	$j < k$, $m \in \{ -1, 1 \}$, $\sigma \in A$, and $f_\sigma^m(a_k) = a_j$. 
	By separation, the triple $(j,m,\sigma)$ is unique. 
	Note that $a_k = f_\sigma^{-m}(a_j)$. 
	By the last bullet point in the assumptions of the proposition, 
	$b_j$ is in the domain of $f_\sigma^{-m}$. 
	Define $b_{k} = f_\sigma^{-m}(b_j)$. 
	Since $a_j <_T b_j$ and $f_\sigma^{-m}$ is strictly increasing, 
	$a_k <_T b_k$. 
	By the inductive hypothesis and 
	the uniqueness of the triple $(j,m,\sigma)$, 
	it easily follows that for all $\tau \in A$, 
	$(a_0,\ldots,a_{k})$ and $(b_0,\ldots,b_{k})$ 
	are $f_\tau$-consistent. 
	So (a) holds. 
	In the case that $k \in \{ i_0, \ldots, i_{l-1} \}$, 
	by the uniqueness of the triple $(j,m,\sigma)$ and the assumption of Case 2, 
	it must be the case that $k = i_{q-1}$ for some $q$ such that 
	$0 < q \le l-1$, $j = i_{q}$, $m = m_{q}$, and $\sigma = \tau_{q}$. 
	By the induction hypothesis, 
	$b_{i_{q}} = c_{q}$, and by the definitions of $b_{i_{q-1}}$ 
	and $c_{q}$, 
	$b_{i_{q-1}} = f_{\tau_{q}}^{-m_{q}}(b_{i_{q}}) = 
	f_{\tau_{q}}^{-m_{q}}(c_{q}) = f_{\tau_q}^{-m_q}(f_{\tau_q}^{m_q}(c_{q-1})) = 
	c_{q-1}$. 
	Thus, (b) holds.
\end{proof}

\section{The Forcing Poset}

We now have the tools at hand to introduce and develop our forcing poset for 
adding a Suslin tree together with $\omega_2$-many automorphisms of it.

\begin{definition}
	Let $\p$ be the forcing poset consisting of conditions 
	which are pairs $(T,F)$ satisfying:
	\begin{enumerate}
	\item $T$ is a standard finite tree;
	\item $F$ is a function whose domain is a finite subset of $\omega_2$, 
	and for all $\gamma \in \dom(F)$, 
	$F(\gamma)$ is a standard function on $T$;
	\item for all $\alpha \in \h[T]$, $F$ 
	is $\rho$-separated on $T_\alpha$.
	\end{enumerate}
	Let $(U,G) \le (T,F)$ if:
	\begin{enumerate}
	\item[(a)] $U$ extends $T$;
	\item[(b)] $\dom(F) \subseteq \dom(G)$ 
	and for all $\gamma \in \dom(F)$, $F(\gamma) \subseteq G(\gamma)$;
	\item[(c)] suppose that $\gamma$ and $\tau$ are distinct elements of $\dom(F)$, 
	$x$ is some element of $U$ in $\dom(G(\gamma)) \cap \dom(G(\tau))$, and 
	$G(\gamma)(x) = G(\tau)(x)$; then there exists some $z \in T$ 
	such that $x \le_U z$ and $F(\gamma)(z) = F(\tau)(z)$.
	\end{enumerate}
\end{definition}

In (3) above, we identify $F$ with the indexed family $\{ F(\tau) : \tau \in \dom(F) \}$ 
(see Definition 4.6).

The first component of a condition in $\p$ approximates a tree whose $\p$-name 
we write as $T^{\dot{G}_{\p}}$, 
and the second component approximates a sequence of functions on $T^{\dot{G}_{\p}}$ 
whose $\p$-names we  write as $F^{\dot{G}_{\p}}_\tau$ for all $\tau < \omega_2$. 
We prove in Section 7 that the forcing poset $\p$ is Knaster and forces that 
$T^{\dot{G}_{\p}}$ is Suslin. 

In this section, we derive some basic lemmas which we use to prove that $\p$ forces that 
$T^{\dot{G}_{\p}}$ is an $\omega$-ary tree with height $\omega_1$ and countable levels and that 
$\{ F^{\dot{G}_{\p}}_\tau : \tau < \omega_2 \}$ is a strongly almost disjoint family 
of automorphisms of $T^{\dot{G}_{\p}}$.

\begin{lemma}
	Suppose that $(T,F) \in \p$, $U$ is a standard finite tree which extends $T$, 
	and $\h[U] \cap (\max(h[T])+1) = \h[T]$. 
	Then $(U,F) \in \p$ and $(U,F) \le (T,F)$. 
\end{lemma}

\begin{proof}
	By Lemma 3.4, for all $\gamma \in \dom(F)$, 
	$F(\gamma)$ is a standard function on $U$. 
	By Lemma 4.7(2), for all $\alpha \in \h[U]$, $F$ is $\rho$-separated on $U_\alpha$. 
	So $(U,F) \in \p$, and it is simple to check that $(U,F) \le (T,F)$.
\end{proof}

\begin{lemma}
	Suppose that $(T,F) \in \p$ 
	and $U$ is a standard finite tree which is a simple extension of $T$. 
	Let $\bar{F}$ be the function with domain equal to $\dom(F)$ such that for all 
	$\tau \in \dom(F)$, $\bar{F}(\tau)$ is the downward closure of $F(\tau)$ in $U$. 
	Then $(U,\bar F) \in \p$ and $(U,\bar F) \le (T,F)$.
\end{lemma}

\begin{proof}
	For proving that $(U,\bar F) \in \p$, we know that $U$ is a standard 
	finite tree by assumption, for all $\tau \in \dom(\bar{F})$, 
	$\bar{F}(\tau)$ is a standard function on $U$ by Lemma 3.8, and 
	for all $\alpha \in \h[U]$, $\bar{F}$ is $\rho$-separated on 
	$U_\alpha$ by Lemma 4.9. 
	For showing that $(U,\bar F) \le (T,F)$, properties (a) and (b) of 
	Definition 5.1 are immediate, and property (c) is easy to check.	
\end{proof}

\begin{lemma}
	Let $Z \subseteq \omega_1 \setminus \{ 0 \}$ be finite. 
	Then the set of $(U,G) \in \p$ such that $Z \subseteq \h[U]$ is dense. 
	In fact, for all $(T,F) \in \p$, there exists $(U,G) \le (T,F)$ such that 
	$U$ is a simple extension of $T$ and  
	$\h[U] = \h[T] \cup Z$.
\end{lemma}

\begin{proof}
	Let $(T,F) \in \p$. 
	By Lemma 2.6, we can find a standard finite tree $U$ which is a simple extension of $T$ 
	such that $\h[U] = \h[T] \cup Z$. 
	Define $G$ with domain equal to $\dom(F)$ so that for each $\gamma \in \dom(G)$, 
	$G(\gamma)$ is the downward closure of $F(\gamma)$ in $U$. 
	By Lemma 5.3, $(U,G) \in \p$ and $(U,G) \le (T,F)$.
\end{proof}

\begin{lemma}
	Let $(T,F) \in \p$ and $x \in T$. 
	Then for any $k > 0$, there exists $(W,H) \le (T,F)$ such that 
	$\h(x)+1 \in \h[W]$ and $|\ISucc_W(x)| \ge k$.
\end{lemma}

\begin{proof}
	By Lemma 5.4, fix $(U,G) \le (T,F)$ such that $U$ is a simple extension of $T$ 
	and $\h(x)+1 \in \h[U]$. 
	If $|\ISucc_U(x)| \ge k$, then we are done. 
	Otherwise, apply Lemma 2.10 (letting $X = \{ x \}$ and $n = k$) 
	to find a standard finite tree $W$ extending $U$ such 
	that $\h[W] = \h[U]$, $W \setminus U \subseteq \ISucc_W(x)$, and $|\ISucc_W(x)| = k$. 
	By Lemma 5.2, $(W,G) \in \p$ and $(W,G) \le (U,G)$.
\end{proof}

\begin{lemma}
	The set of conditions $(U,G) \in \p$ such that $U$ is Hausdorff is dense.
\end{lemma}

\begin{proof}
	Let $(T,F) \in \p$. 
	For each limit ordinal $\delta \in \h[T]$, 
	fix a successor ordinal $\delta^-$ such that 
	$\max(\h[T] \cap \delta) < \delta^- < \delta$. 
	Let 
	$$
	Z = \{ \delta^- : \delta \in \h[T], \ \delta \ \text{is a limit ordinal} \}
	$$
	Apply Lemma 5.4 to find $(U,G) \le (T,F)$ such that 
	$U$ is a simple extension of $T$ and $\h[U] = \h[T] \cup Z$. 
	To show that $U$ is Hausdorff, consider distinct $x$ and $y$ in $U_\delta$, 
	where $\delta \in \h[U]$ is a limit ordinal. 
	Since $\h[U] \setminus \h[T]$ consists of successor ordinals, $\delta \in \h[T]$, 
	so $x$ and $y$ are in $T_\delta$. 
	As $U$ is a simple extension of $T$, $x \res \delta^- \ne y \res \delta^-$.
\end{proof}

\begin{lemma}
	The set of conditions $(U,G) \in \p$ such that $U$ is normal is dense.
\end{lemma}

\begin{proof}
	Let $(T,F) \in \p$. 
	By Lemma 2.9, we can 
	fix a standard finite tree $U$ which is normal such that $U$ 
	extends $T$ and $\h[U] = \h[T]$. 
	By Lemma 5.2, $(U,F) \in \p$ and $(U,F) \le (T,F)$.
\end{proof}

\begin{lemma}
	For any $(T,F) \in \p$, $x \in T$, and countable $\alpha > \h(x)$, 
	there exists $(W,H) \le (T,F)$ such that $\alpha \in \h[W]$ 
	and for some $y \in W_\alpha$, $x <_W y$.
\end{lemma}

\begin{proof}
	By Lemma 5.4, we can fix $(U,G) \le (T,F)$ such that $\alpha \in \h[U]$. 
	Now apply Lemma 5.7 to find $(W,H) \le (U,G)$ such that $W$ is normal. 
\end{proof}

\begin{definition}
	For any generic filter $G_{\p}$ on $\p$, let $T^{{G_{\p}}}$ be the tree with underlying set 
	$\bigcup \{ T : \exists F \ (T,F) \in {G_{\p}} \}$, and ordered by $x <_{{{G_{\p}}}} y$ 
	if there exists some $(T,F) \in {{G_{\p}}}$ such that $x <_T y$. 
	Let $T^{\dot{G}_{\p}}$ be a $\p$-name for this object.
\end{definition}

We refer to a tree with height $\omega_1$ and countable levels as an \emph{$\omega_1$-tree}. 
An $\omega_1$-tree is \emph{normal} if it has a root, every element has at least two immediate 
successors, it is Hausdorff, and every element has elements above it at every higher level 
(note that this is not exactly the same 
as the same definition of normal for standard finite trees). 
By $\omega$-ary, we mean that every element of the tree has $\omega$-many 
immediate successors.

\begin{proposition}
	The forcing poset $\p$ forces that, assuming $\omega_1$ is preserved, 
	$T^{\dot{G}_{\p}}$ is an $\omega_1$-tree which is $\omega$-ary and normal.
\end{proposition}

\begin{proof}
Let $G_{\p}$ be a generic filter on $\p$. 
It is easy to prove that $(T^{{G_{\p}}},<_{{G_{\p}}})$ is a tree (the well-foundedness follows 
from the fact that $x <_{{G_{\p}}} y$ implies that $\h(x) < \h(y)$). 
Lemma 5.4 implies that the height function on $T^{{G_{\p}}}$ coincides with 
the height function $\h$ that we defined on countable ordinals. 
It then follows that the levels 
of $T^{{G_{\p}}}$ are countable and $T^{{G_{\p}}}$ has height $\omega_1^V$. 
For being $\omega$-ary and normal, 
clearly $0$ is the root of $T^{{G_{\p}}}$, Lemma 5.5 implies that $T^{{G_{\p}}}$ is $\omega$-ary, 
Lemma 5.6 implies that $T^{{G_{\p}}}$ is Hausdorff, and Lemma 5.8 completes the proof.
\end{proof}

We prove in Section 7 that $\p$ is Knaster, and hence $\p$ preserves $\omega_1$. 
We also prove in Section 7 that $\p$ forces that $T^{\dot{G}_{\p}}$ is Suslin.

\begin{lemma}
	Suppose that $(T,F) \in \p$ and $\sigma < \omega_2$. 
	Then there exists $(U,G) \le (T,F)$ such that $\sigma \in \dom(G)$.
\end{lemma}

\begin{proof}
	If $\sigma \in \dom(F)$, then we are done. 
	Otherwise, define $G$ with domain equal to $\dom(F) \cup \{ \sigma \}$, where 
	$G \res \dom(F) = F$ and $G(\sigma) = \emptyset$. 
	It is simple to check that $(T,G) \in \p$ and $(T,G) \le (T,F)$.
\end{proof}

\begin{lemma}[Augmentation]
	Suppose that $(T,F) \in \p$, $\sigma < \omega_2$, and $x \in T$. 
	Then there exists $(U,G) \le (T,F)$ such that $\sigma \in \dom(G)$ and 
	$x$ is in the domain and range of $G(\sigma)$.
\end{lemma}

\begin{proof}
	By Lemma 5.11, 
	without loss of generality we may assume that $\sigma \in \dom(F)$. 
	We prove the statement about $x$ being in the domain by induction 
	on the height of $x$; the proof for the range is similar. 
	For the base case, if $x$ has height $0$, then $x = 0$. 
	So we can extend $(T,F)$ as required by mapping $0$ to $0$. 
	Now suppose that $x$ has height $\alpha > 0$. 
	If $x \in \dom(F(\sigma))$ then we are done, so assume not. 
	Let $x^-$ be the largest member of $\{ z \in T : z <_T x \}$. 
	By the inductive hypothesis, we may assume without loss of generality 
	that $x^- \in \dom(F(\sigma))$. 
	Extend $T$ to $U$ by adding a single element $z$ which is not in $T$ above 
	$F(\sigma)(x^-)$ with height $\alpha$. 
	Define $G$ with the same domain as $F$ with the only change being that $G(\sigma)(x) = z$.
	
	To prove that $(U,G)$ is a condition, we only show that 
	$G$ is $\rho$-separated on $U_\alpha$ since the other properties are clear. 
	Let $\vec a$ be an injective tuple which lists $T_\alpha$ so that 
	$F$ is $\rho$-separated on $\vec a$. 
	Since the only relation which $z$ has with members of $\vec a$ 
	is the equation $G(\sigma)^{-1}(z) = x$, 
	clearly $G$ is $\rho$-separated on $\vec a^{\frown}z$. 
	Hence, $(U,G) \in \p$, and easily $(U,G) \le (T,F)$.
\end{proof}

\begin{definition}
	For any generic filter ${G_{\p}}$ on $\p$ and $\tau < \omega_2$, 
	define 
	$$
	F_\tau^{{G_{\p}}} = \bigcup \{ F(\tau) : \exists T \ (T,F) \in {G_{\p}} \ 
	\text{and} \ \tau \in \dom(F) \}.
	$$
	Let $F_\tau^{\dot{G}_{\p}}$ be a $\p$-name for this object.
\end{definition}

\begin{proposition}
	The forcing poset $\p$ forces that for all $\tau < \omega_2$, 
	$F^{\dot{G}_{\p}}_\tau$ is an automorphism of $T^{\dot{G}_{\p}}$.
\end{proposition}

\begin{proof}
	A straightforward argument shows that $F^{\dot{G}_{\p}}_\tau$ is forced to be 
	a function, and by Lemma 5.12, it is total and surjective. 
	Whenever $(T,F) \in \p$ and $\tau \in \dom(F)$, 
	$F(\tau)$ is injective and strictly increasing. 
	It easily follows that $F^{\dot{G}_{\p}}_\tau$ is forced to be injective and strictly increasing. 
	So $F^{\dot{G}_{\p}}_\tau$ is forced to be a strictly increasing bijection, and hence an 
	automorphism.
\end{proof}

\begin{proposition}
	The forcing poset $\p$ forces that for all $\gamma < \tau < \omega_2$, 
	$F^{\dot{G}_{\p}}_\gamma$ and $F^{\dot{G}_{\p}}_\tau$ are strongly almost disjoint subsets 
	of $T^{\dot{G}_{\p}} \otimes T^{\dot{G}_{\p}}$.
\end{proposition}

\begin{proof}
	Let ${G_{\p}}$ be a generic filter on $\p$. 
	By Lemma 5.11, we can fix a condition $(T,F) \in {G_{\p}}$ such that $\gamma$ and $\tau$ 
	are in $\dom(F)$. 
	Suppose that $F^{{G_{\p}}}_\gamma(x) = y$ and $F^{{G_{\p}}}_\tau(x) = y$. 
	Then clearly there exists some $(W,H) \in {G_{\p}}$ such that $(W,H) \le (T,F)$, 
	$H(\gamma)(x) = y$, and $H(\tau)(x) = y$. 
	By the definition of the ordering of $\p$, 
	there exist $z$ and $z^*$ in $T$ 
	such that $x \le_W z$, $F(\gamma)(z) = z^*$, and $F(\tau)(z) = z^*$. 
	Hence, in the tree $T^{{G_{\p}}} \otimes T^{{G_{\p}}}$, $(x,y) \le (z,z^*)$. 
	It follows that $F^{{G_{\p}}}_\gamma \cap F^{{G_{\p}}}_\tau$ is a subset of the downward closure 
	of the finite set $\{ (a,b) \in T \otimes T : F(\gamma)(a) = b \}$, and therefore 
	is a finite union of countable chains.
\end{proof}

In Section 7 we prove that $\p$ is Knaster and forces that $T^{\dot{G}_{\p}}$ is Suslin. 
These two facts combined with Proposition 5.10 and Propositions 5.14 and 5.15 complete  
the proof of the main theorem.

\section{Making the Functions Bijective}

In order to prove that $\p$ forces that 
the generic tree $T^{{\dot{G}_{\p}}}$ is Suslin, we need to apply 
Proposition 4.11 ($1$-Key Property). 
The main challenge in doing this is to construct a condition 
which satisfies the assumption given in the fourth bullet point of that proposition. 
Namely, we need to extend a condition so that 
some of its functions which are separated 
on a subset of some level of its tree are total and surjective above that subset. 
In this section we  achieve this goal.

\begin{lemma}
	Suppose that $(T,F) \in \p$, $\alpha \in \h[T] \cap \max(\h[T])$, 
	$X \subseteq T_\alpha$, and $n$ is a positive natural number such that 
	every element of $X$ has at most $n$-many immediate successors. 
	Then there exists $U$ such that:
	\begin{enumerate}
	\item $(U,F) \in \p$ and $(U,F) \le (T,F)$;
	\item $\h[T] = \h[U]$;
	\item $U \setminus T \subseteq \bigcup \{ \mathrm{ISucc}_U(x) : x \in X \}$;
	\item every element of $X$ has exactly $n$-many immediate successors in $U$.
	\end{enumerate}
\end{lemma}

\begin{proof}
	Apply Lemma 2.10 to find a standard finite tree $U$ which extends $T$ and satisfies 
	(2), (3), and (4). 
	By Lemma 5.2, $(U,F) \in \p$ and $(U,F) \le (T,F)$.
\end{proof}

\begin{proposition}
	Suppose that $(T,F) \in \p$, 
	$\alpha \in \h[T] \cap \max(\h[T])$, 
	$\beta = \min(\h[T] \setminus (\alpha+1))$, 
	$X \subseteq T_\alpha$ is non-empty, and $A \subseteq \dom(F)$. 
	Assume that $\{ F(\tau) : \tau \in A \}$ is separated on $X$. 
	Then there exists $(U,G) \in \p$ satisfying:
	\begin{itemize}
		\item $(U,G) \le (T,F)$;
		\item $\h[T] = \h[U]$ and $\dom(F) = \dom(G)$;
		\item $U \setminus T \subseteq \bigcup \{ \ISucc_U(x) : x \in X \}$;
		\item for all $x \in X$, $\ISucc_U(x)$ is non-empty;
		\item for all $\tau \in \dom(G)$ and for all 
		$z \in \dom(G(\tau)) \setminus \dom(F(\tau))$, 
		both $z$ and $G(\tau)(z)$ are in $\bigcup \{ \ISucc_U(x) : x \in X \}$;
		\item $\{ G(\tau) : \tau \in A \}$ is separated on 
		$\bigcup \{ \ISucc_U(x) : x \in X \}$;
		\item for all $\tau \in A$ and for all $x, y \in X$, if 
		$x \in \dom(G(\tau))$ and $G(\tau)(x) = y$, 
		then $\ISucc_U(x) \subseteq \dom(G(\tau))$ and 
		$\ISucc_U(y) \subseteq \ran(G(\tau))$.
	\end{itemize}
\end{proposition}

\begin{proof}
	Let $q = |X|$. 
	Fix an injective tuple $(a_0,\ldots,a_{q-1})$ which lists the elements of $X$ so that 
	$\{ F(\tau) : \tau \in A \}$ is separated on $(a_0,\ldots,a_{q-1})$. 
	Choose a natural number $p > 0$ such that every element of $X$ has 
	at most $p$-many immediate successors. 
	Applying Lemma 6.1, fix $U$ such that $(U,F) \in \p$, $(U,F) \le (T,F)$, 
	$\h[T] = \h[U]$, $U \setminus T \subseteq \bigcup \{ \ISucc_U(x) : x \in X \}$, 
	and every element of $X$ has exactly $pq$-many immediate successors in $U$. 
	Let $Y = \bigcup \{ \ISucc_U(x) : x \in X \}$. 
	For each $i < q$, fix a partition $\{ X^i_0,\ldots,X^i_{q-1} \}$ of $\ISucc_U(a_i)$ 
	into disjoint sets each of size $p$ and satisfying that 
	$\ISucc_T(a_i) \subseteq X^i_i$.
	
	We define a function $G$ with domain equal to $\dom(F)$ 
	satisfying the following properties for each $\tau \in \dom(F)$:
	\begin{itemize}
		\item[(I)] $G(\tau)$ is a standard function on $U$ and 
		$F(\tau) \subseteq G(\tau)$;
		\item[(II)] if $\tau \notin A$, then $G(\tau) = F(\tau)$;
		\item[(III)] if $\tau \in A$, $G(\tau)(x) = y$, and $x \notin \dom(F(\tau))$, 
		then $x$ and $y$ are in $Y$ and $F(\tau)(x \res \alpha) = y \res \alpha$.
	\end{itemize}

	Let us see what conclusions can be drawn from (I)--(III). 
	Note that for all $\tau \in \dom(F)$ and for all 
	$\delta \in \h[U]$ different from $\beta$, 
	$G(\tau) \res U_\delta = F(\tau) \res U_\delta$. 
	In particular, $\{ G(\tau) : \tau \in A \}$ is separated on $(a_0,\ldots,a_{q-1})$. 
	Concerning showing that $(U,G) \in \p$: Definition 5.1(1,2) are clear, and 
	Definition 5.1(3) holds provided that $G$ is $\rho$-separated on $U_\beta$. 
	For showing that $(U,G) \le (U,F)$, Definition 5.1(a,b) are clear, so it 
	suffices to prove Definition 5.1(c). 
	Assume for a moment that $(U,G) \in \p$ and $(U,G) \le (U,F)$, 
	and let us review the conclusions of the proposition. 
	The first five bullet points of the proposition are clear, and the sixth 
	follows from the fact that $\{ G(\tau) : \tau \in A \}$ is separated on $X$ 
	and Lemma 4.5 (Strong Persistence). 
	So for verifying the conclusions of the proposition, 
	it suffices to prove the seventh bullet point. 

	To summarize, in order to complete the proof of the proposition, 
	it suffices to define a function $G$ with domain equal to $\dom(F)$ 
	which satisfies (I), (II), and (III) above for all $\tau \in \dom(F)$, and 
	also has the following properties:
	\begin{itemize}
	\item[(IV)] $G$ is $\rho$-separated on $U_\beta$;
	\item[(V)] for all distinct $\gamma$ and $\tau$ in $\dom(F)$, if 
	$x \in \dom(G(\gamma)) \cap \dom(G(\tau))$ and 
	$G(\gamma)(x) = G(\tau)(x)$, then there exists some $z \in U$ 
	such that $x \le_U z$ and $F(\gamma)(z) = F(\tau)(z)$;
	\item[(VI)] for all $\tau \in A$ and for all $x, y \in X$, if 
	$x \in \dom(G(\tau))$ and $G(\tau)(x) = y$, 
	then $\ISucc_U(x) \subseteq \dom(G(\tau))$ and 
	$\ISucc_U(y) \subseteq \ran(G(\tau))$.
	\end{itemize}

	Let us now proceed with the definition of $G$. 
	For all $\tau \in \dom(F) \setminus A$, define $G(\tau) = F(\tau)$. 
	So (II) is satisfied. 
	Now consider $\tau \in A$. 
	Define $G(\tau) \res \dom(F(\tau)) = F(\tau)$. 
	Consider $x \in U \setminus \dom(F(\tau))$. 
	We let $x$ be in the domain of $G(\tau)$ if and only if 
	$x \res \alpha \in X \cap \dom(F(\tau))$ and $F(\tau)(x \res \alpha) \in X$.  
	In that case, 
	we define $G(\tau)(x)$ to be an element of $\ISucc_U(F(\tau)(x \res \alpha))$, 
	as described below. 
	In other words, for any $i, j < q$ such that $F(\tau)(a_i) = a_j$, 
	we  define $G(\tau)$ on $\ISucc_U(a_i)$ so that it is a bijection 
	between $\ISucc_U(a_i)$ and $\ISucc_U(a_j)$. 
	Let us assume for a moment that we succeed in defining $G(\tau)$ in this manner. 
	It is routine to verify that $G(\tau)$ is a standard function on $U$, 
	so (I) holds. 
	Also, (III) and (VI) are clear.
	
	So let $i, j < q$ be given such that $F(\tau)(a_i) = a_j$. 
	Since $F(\tau)$ is a standard function, $i \ne j$. 
	Recall that $\{ X^i_0,\ldots,X^i_{q-1} \}$ is a partition of $\ISucc_U(a_i)$ 
	into disjoint sets each of size $p$ and satisfying that 
	$$
	\dom(F(\tau)) \cap \ISucc_U(a_i) \subseteq \ISucc_T(a_i) \subseteq X^i_i.
	$$
	Similarly, $\{ X^j_0,\ldots,X^j_{q-1} \}$ is a partition of $\ISucc_U(a_j)$ 
	into disjoint sets each of size $p$ and satisfying that 
	$$
	\ran(F(\tau)) \cap \ISucc_U(a_j) \subseteq \ISucc_T(a_j) \subseteq X^j_j.
	$$

	We define $G(\tau)$ extending $F(\tau)$ 
	which maps $\ISucc_U(a_i)$ bijectively onto $\ISucc_U(a_j)$ 
	and has the following properties:
	\begin{itemize}
		\item[(a)] for all $k \in q \setminus \{ i, j \}$, 
		$G(\tau)[X^i_k] = X^j_k$;
		\item[(b)] $G(\tau)[X^i_i \setminus \dom(F(\tau))] \subseteq X^j_i$;
		\item[(c)] $G(\tau)^{-1}[X^j_j \setminus \dom(F(\tau)^{-1})] 
		\subseteq X^i_j$.
	\end{itemize}
	Let us see how we can arrange this. 
	By the last paragraph, 
	since $G(\tau) \res \dom(F(\tau)) = F(\tau)$, 
	we have specified $G(\tau)(x)$ for 
	those $x$ which are in the set 
	$\dom(F(\tau)) \cap \ISucc_U(a_i)$, 
	and hence are in $X^i_i$, and the values 
	$G(\tau)(x)$ for such $x$ are in $X^j_j$. 
	For each $k < q$ different from $i$ and $j$, 
	$X^i_k$ and $X^j_k$ have the same size $p$, so we can easily arrange that (a) holds. 
	This defines $G(\tau)$ on $\ISucc_U(a_i) \setminus (X^i_i \cup X^i_j)$ and 
	on $\dom(F(\tau)) \cap X^i_i$. 
	We have not defined $G(\tau)(x)$ yet for any $x \in X^i_j$. 
	Also, the values of $G(\tau)$ which we have defined so far are not in $X^j_i$. 
	Define $G(\tau)$ on $X^i_i \setminus \dom(F(\tau))$ so that it maps 
	injectively into $X^j_i$. 
	This completes the definition of $G(\tau)$ on $X^i_i$ and it satisfies (b). 
	It remains to define $G(\tau)$ on $X^i_j$ so that (c) is satisfied. 
	We have not defined any values of $G(\tau)$ on $X^i_j$ yet, so we can do so in 
	such a way that $G(\tau)^{-1}$ maps $X^j_j \setminus \ran(F(\tau))$ into $X^i_j$. 
	Hence, (c) is satisfied. 
	Now extend the definition of $G(\tau)$ injectively to the rest of $X^i_j$, 
	mapping to values not already taken (the remaining values of $G(\tau)$ on 
	$X^i_j$ besides those specified above must be in $X^j_i$, since all of the 
	other values are already taken, although that fact does not matter for us).

	This completes the definition of $G$. 
	It remains to prove (IV) and (V).
		
	(V) Let $\gamma$ and $\tau$ be distinct elements of $\dom(F)$, 
	and assume that $x \in \dom(G(\gamma)) \cap \dom(G(\tau))$ 
	and $G(\gamma)(x) = G(\tau)(x)$. 
	We prove that there exists some $z \in U$ 
	such that $x \le_U z$ and $F(\gamma)(z) = F(\tau)(z)$. 
	If $x \in \dom(F(\gamma)) \cap \dom(F(\tau))$, then we are done by letting $z = x$. 
	Otherwise, we can assume without loss of generality that $x \notin \dom(F(\gamma))$. 
	Then by (II) and (III), $\gamma \in A$ and both $x$ and $G(\gamma)(x)$ are in $Y$. 
	Fix $i, j < q$ such that $x \in \ISucc_U(a_i)$ and $G(\gamma)(x) \in \ISucc_U(a_j)$.
	
	If $\tau \in A$, then by Lemma 4.4 the equation 
	$G(\tau)(x) = G(\gamma)(x)$ contradicts the fact that 
	$\{ G(\xi) : \xi \in A \}$ is separated on $Y$. 
	So $\tau \notin A$. 
	By (II), $G(\tau) = F(\tau)$, so $G(\gamma)(x) = F(\tau)(x)$. 
	Consequently, $x$ and $G(\gamma)(x)$ are in $T$. 
	By the choice of the partitions, 
	$x \in X^i_i$ and $G(\gamma)(x) \in X^j_j$. 
	Since $x \notin \dom(F(\gamma))$, (b) implies that $G(\gamma)(x)$ 
	is in $X^j_i$, which is a contradiction since $X^j_j$ and $X^j_i$ are disjoint.  
	This completes the proof of (V).

	(IV) We now begin the proof that $G$ is $\rho$-separated on $U_\beta$. 
	We need the following three claims. 
	
	\underline{Claim 1:} Let $i, j < q$, $x \in \ISucc_T(a_i)$, 
	$y \in \ISucc_U(a_j) \setminus T$, $\tau \in A$, and $m \in \{ -1, 1 \}$. 
	Suppose that $G(\tau)^m(x) = y$. 
	Then $y \in X^j_i$. 

	\emph{Proof.} Since $y \notin T$, $x \notin \dom(F(\tau)^m)$. 
	On the other hand, $x \in T$, so $x \in X^i_i \setminus \dom(F(\tau)^m)$. 
	By (b) and (c), $y \in X^j_i$. \qedsymbol

	\underline{Claim 2:} Let $i, j, k < q$ and let $\tau$ and $\sigma$ be in $\dom(F)$. 
	Suppose that $x, y, z \in U$ are distinct, 
	$x \in \ISucc_T(a_i)$, $y \in \ISucc_U(a_j) \setminus T$, $z \in \ISucc_U(a_k)$, 
	$m, n \in \{ -1, 1 \}$, 
	$G(\tau)^m(x) = y$, and $G(\sigma)^n(y) = z$. 
	Then $z \notin T$. 

	\emph{Proof.} Since $y \notin T$, 
	$y \notin \ran(F(\tau)^{m})$ and $y \notin \dom(F(\sigma)^{n})$. 
	By (II), it follows that $\tau \in A$ and $\sigma \in A$. 		
	Applying Claim 1 for $x$ and $y$, we get that $y \in X^j_i$. 
	Suppose for a contradiction that $z \in T$. 
	Applying Claim 1 to $z$ and $y$ and the equation 
	$G(\sigma)^{-n}(z) = y$, we get that $y \in X^j_k$. 
	Consequently, $i = k$. 
	Since $G(\tau)^{-m}(y) = x$ and $G(\sigma)^{n}(y) = z$, it follows that 
	$F(\tau)^{-m}(a_j) = a_i$ and $F(\sigma)^{n}(a_j) = a_i$. 
	As $\{ F(\xi) : \xi \in A \}$ is separated on $X$, 
	$\tau = \sigma$ and $-m = n$. 
	But then $x = z$, which is a contradiction. \qedsymbol

	\underline{Claim 3:} There do not exist tuples 
	$(b_0,\ldots,b_{l-1})$, $(m_0,\ldots,m_{l-2})$, and $(\tau_0,\ldots,\tau_{l-2})$, 
	for some natural number $l \ge 3$, satisfying: 
	\begin{itemize}
	\item for all $j < l$, $b_j \in Y$;
	\item $m_0,\ldots,m_{l-2}$ are in $\{ -1, 1 \}$ and 
	$\tau_0,\ldots,\tau_{l-2}$ are in $A$;
	\item $G(\tau_j)^{m_j}(b_j) = b_{j+1}$ for all $j < l-1$;
	\item $b_0 \in T$, $b_j \notin T$ for all $0 < j < l-1$, and $b_{l-1} \in T$;
	\item $b_0,\ldots,b_{l-2}$ are all distinct;
	\item $b_0$, $b_1$, and $b_2$ are distinct.
	\end{itemize}

	\emph{Proof.} Suppose for a contradiction that such tuples exist. 
	Fix $i_0,\ldots,i_{l-1} < q$ such that 
	$b_j \in \ISucc_U(a_{i_j})$ for all $j < l$. 
	Note that $F(\tau_j)^{m_j}(a_{i_j}) = a_{i_{j+1}}$ for all $j < l-1$. 
	By Claim 2 applied to $b_0$, $b_1$, and $b_2$, we can conclude that $b_2$ is not in $T$. 
	Since $b_{l-1} \in T$, it follows that $l \ge 4$.

	We claim that for all $j < k < l-1$, $i_j \ne i_k$. 
	Otherwise, let $k < l-1$ be least such that for some $j < k$, 
	$i_j = i_k$. 
	Since $F(\tau_{k-1})$ has no fixed points in $X$ and 
	$F(\tau_{k-1})^{m_{k-1}}(a_{i_{k-1}}) = a_{i_k}$, $j < k - 1$. 
	If $j = k - 2$, then we would have that 
	$F(\tau_j)^{-m_{j}}(a_{i_{j+1}}) = a_{i_j}$ and 
	$$
	F(\tau_{j+1})^{m_{j+1}}(a_{i_{j+1}}) = a_{i_{j+2}} = a_{i_{k}} = a_{i_j},
	$$
	which implies by Lemma 4.4 that 
	$\tau_{j} = \tau_{j+1}$ and $-m_j = m_{j+1}$. 
	So $G(\tau_j)^{-m_j}(b_{j+1}) = b_j$ would be equal to 
	$G(\tau_{j+1})^{m_{j+1}}(b_{j+1}) = b_{j+2} = b_k$, which contradicts 
	that $b_j \ne b_k$. 
	Hence, $j < k-2$. 
	But then $\langle a_{i_j},a_{i_{j+1}},\ldots,a_{i_k} \rangle$ is a loop with respect 
	to $\{ F(\xi) : \xi \in A \}$. 
	By Proposition 4.10 (Characterization of $\rho$-Separation), 
	$\{ F(\xi) : \xi \in A \}$ is not $\rho$-separated on $X$, and hence is not 
	separated on $X$, which is a contradiction.

	Since $b_0 \in T$, $b_0 \in X^{i_0}_{i_0}$. 
	But $b_0 \notin \dom(F(\tau_0)^{m_0})$, for otherwise $b_1$ would be in $T$. 
	By (b) and (c) it follows that $b_1 \in X^{i_1}_{i_0}$. 
	Using the fact that $i_0,i_1,\ldots,i_{l-2}$ are all distinct, 
	it follows by induction using (a) that for all $k < l-1$, 
	$b_k \in X^{i_k}_{i_0}$. 
	In particular, $b_{l-2} \in X^{i_{l-2}}_{i_0}$. 
	By Claim 1 (letting $x = b_{l-1}$ and $y = b_{l-2}$), 
	$b_{l-2} \in X^{i_{l-2}}_{i_{l-1}}$. 
	It follows that $i_{l-1} = i_0$.
	So $a_{i_{l-1}} = a_{i_0}$. 
	Since $a_0,\ldots,a_{i_{l-2}}$ are distinct and $l \ge 4$, 
	$\langle a_{i_0},\ldots,a_{i_{l-2}},a_{i_{l-1}} \rangle$ 
	is a loop with respect to $\{ F(\xi) : \xi \in A \}$. 
	By Proposition 4.10 (Characterization of $\rho$-Separation), 
	$\{ F(\xi) : \xi \in A \}$ is not $\rho$-separated on $X$, and hence is not 
	separated on $X$, which is a contradiction. \qedsymbol

	Now we are ready to prove that $G$ is $\rho$-separated on $U_\beta$. 
	We verify properties (1) and (2) of 
	Proposition 4.10 (Characterization of $\rho$-Separation). 
	(1) Since $(U,F) \in \p$, $F$ is $\rho$-separated on $U_\beta$. 
	Suppose that $c, d \in U_\beta$, 
	$(m,\tau)$ and $(n,\sigma)$ are distinct elements of 
	$\{ -1, 1 \} \times \dom(G)$, 
	and $G(\tau)^m(c) = d$ and $G(\sigma)^n(c) = d$. 
	We show that $\rho(\tau,\sigma) \ge \beta$. 
	If $c \in \dom(F(\tau)^m)$ and $c \in \dom(F(\sigma)^n)$, then 
	$F(\tau)^m(c) = d$ and $F(\sigma)^n(c) = d$, so $\rho(\tau,\sigma) \ge \beta$ 
	since $F$ is $\rho$-separated on $U_\beta$. 
	Otherwise, without loss of generality $c \notin \dom(F(\tau)^m)$. 
	By (II) and (III), $\tau \in A$ and there are distinct $i, j < q$ such that 
	$c \in \ISucc_U(a_i)$ and $d \in \ISucc_U(a_j)$. 
	If $\sigma \in A$, then by Lemma 4.4 
	we have a contradiction to the fact that 
	$\{ G(\xi): \xi \in A \}$ is separated on $Y$. 
	So $\sigma \notin A$. 
	Hence, $F(\sigma)^n(c) = d$. 
	So $c$ and $d$ are both in $T$, and therefore $c \in X^i_i$ and $d \in X^j_j$. 
	On the other hand, $c \notin \dom(F(\tau)^m)$, so by (b) and (c), 
	$G(\tau)^m(c) = d$ is in $X^j_i$, which contradicts the fact that 
	$X^j_j$ and $X^j_i$ are disjoint.

	(2) Suppose for a contradiction 
	that there exists a loop $\langle b_0,\ldots,b_{k-1} \rangle$ in $U_\beta$ 
	with respect to $G$. 
	So $k \ge 4$, there exist $\tau_0,\ldots,\tau_{k-2}$ in $\dom(G)$ and 
	$m_0,\ldots,m_{k-2}$ in $\{ -1, 1 \}$ such that 
	$b_0 = b_{k-1}$, $\langle b_0,\ldots,b_{k-2} \rangle$ is injective, 
	and for all $i < k-1$, $G(\tau_i)^{m_i}(b_i) = b_{i+1}$. 
	If for all $i < k-1$, $b_i \in \dom(F(\tau_i)^{m_i})$, then we have a 
	contradiction to the fact that $F$ is $\rho$-separated on $U_\beta$. 
	If for all $i < k-1$, $\tau_i \in A$ and $b_i \in Y$, then we have a 
	contradiction to the fact that $\{ G(\xi) : \xi \in A \}$ is separated on $Y$.

	So we may assume that (i) for some $i < k - 1$, $b_i \notin \dom(F(\tau_i)^{m_i})$, 
	and (ii) for some $j < k - 1$, either $\tau_j \notin A$ or $b_j \notin Y$. 
	Note that it follows from (i), (II), and (III) that 
	$b_i$ and $b_{i+1}$ are in $Y$. 
	For (ii), in either case $b_j \in T$ (namely, use (II) if $\tau_j \notin A$ and 
	the fact that $U \setminus T \subseteq Y$ if $b_j \notin Y$). 
	We claim that in (i), either $b_i$ or $b_{i+1}$ is not in $T$. 
	If $b_i \notin T$, then we are done, so assume that $b_i \in T$. 
	Fix $s, t < q$ such that $b_i \in \ISucc_U(a_s)$ and $b_{i+1} \in \ISucc_U(a_t)$. 
	Then $s \ne t$. 
	Since $b_i \in T$, $b_i \in X^s_s \setminus \dom(F(\tau_i)^{m_i})$. 
	By (b) and (c), $b_{i+1} \in X^t_s$, and hence $b_{i+1} \notin T$ since 
	otherwise $b_{i+1}$ would be in $X^t_t$.
	
	In conclusion, some member of the loop is not in $T$ and some member is in $T$. 
	Obviously, this implies that there are adjacent members of the loop where 
	one is in $T$ and the other is not in $T$. 
	By shifting the loop if necessary, we may assume without loss of generality 
	that $b_0 \in T$ and $b_1 \notin T$. 
	Then $b_{k-1} = b_0$ is in $T$. 
	Let $l \le k$ be the least natural number greater than $1$ such that 
	$b_{l-1} \in T$. 
	Since $b_1 \notin T$, by (II) and (III) it follows that $\tau_0 \in A$ and $b_0 \in Y$. 
	Similarly, for all $0 < s < l-1$, $b_s \notin T$ 
	implies by (II) and (III) 
	that $\tau_s \in A$ and both $b_s$ and $b_{s+1}$ are in $Y$. 
	In particular, $b_{l-1} \in Y$. 
	Since $k \ge 4$, $b_0$, $b_1$, and $b_2$ are all different. 
	As $b_1 \notin T$, $l \ge 3$. 
	Using this information, 
	it is easy to check that $(b_0,\ldots,b_{l-1})$, 
	$(\tau_0,\ldots,\tau_{l-2})$, and $(m_0,\ldots,m_{l-2})$ 
	satisfy the description of the tuples which Claim 3 states does not exist, 
	which is a contradiction.
	\end{proof}

\begin{corollary}
	Suppose that $(T,F) \in \p$, 
	$\alpha \in \h[T]$, $X \subseteq T_\alpha$ is non-empty, 
	and $A \subseteq \dom(F)$. 
	Assume that $\{ F(\tau) : \tau \in A \}$ is separated on $X$. 
	Then there exists $(U,G) \in \p$ satisfying:
	\begin{itemize}
		\item $(U,G) \le (T,F)$;
		\item $\h[T] = \h[U]$ and $\dom(F) = \dom(G)$;
		\item $U \setminus T \subseteq \bigcup \{ \Succ_U(x) : x \in X \}$;
		\item for all $\tau \in \dom(G)$ and for all 
		$z \in \dom(G(\tau)) \setminus \dom(F(\tau))$, 
		both $z$ and $G(\tau)(z)$ are in $\bigcup \{ \Succ_U(x) : x \in X \}$;
		\item for all $\tau \in A$ and for all $x, y \in X$, if 
		$x \in \dom(G(\tau))$ and $G(\tau)(x) = y$, 
		then $\Succ_U(x) \subseteq \dom(G(\tau))$ and 
		$\Succ_U(y) \subseteq \ran(G(\tau))$.
	\end{itemize}
\end{corollary}

\begin{proof}
	By induction on the ordinals in $\h[T] \setminus (\alpha+1)$, 
	we can build the desired condition in finitely many steps, where at each step 
	we use Proposition 6.2 to go up one more level.
\end{proof}

\begin{corollary}
	Suppose that $(T,F) \in \p$, 
	$\alpha \in \h[T]$, $X \subseteq T_\alpha$ is non-empty, 
	and $A \subseteq \dom(F)$. 
	Assume that $\{ F(\tau) : \tau \in A \}$ is separated on $X$. 
	Let $\beta = \max(\h[T])$ and let $b \in T_\beta$ be such that $b \res \alpha \in X$. 
	Then there exists $(U,G) \in \p$ and $Y \subseteq U_\beta$ satisfying:
	\begin{enumerate}
		\item $(U,G) \le (T,F)$;
		\item $\h[T] = \h[U]$ and $\dom(F) = \dom(G)$;
		\item $U \setminus T \subseteq \bigcup \{ \Succ_U(x) : x \in X \}$;
		\item for all $\tau \in \dom(G)$ and for all 
		$z \in \dom(G(\tau)) \setminus \dom(F(\tau))$, 
		both $z$ and $G(\tau)(z)$ are in $\bigcup \{ \Succ_U(x) : x \in X \}$;
		\item $Y$ has unique drop-downs to $\alpha$, $Y \res \alpha = X$, and $b \in Y$;	
		\item for all $\tau \in A$, $X$ and $Y$ are $G(\tau)$-consistent.
	\end{enumerate}
\end{corollary}

\begin{proof}
	Immediate from Proposition 4.11 ($1$-Key Property) and Corollary 6.3.
\end{proof}

\section{The Generic Tree is Suslin}

We now complete the proof of the main theorem by showing that $\p$ is Knaster and 
forces that $T^{{\dot{G}_{\p}}}$ is Suslin. 
These facts follow from the next theorem (see Corollaries 7.2 and 7.3).

\begin{thm}
	Suppose that $\langle (T^\alpha,F^\alpha) : \alpha < \omega_1 \rangle$ 
	is a sequence of conditions in $\p$ and for each $\alpha < \omega_1$, 
	$x^\alpha \in T^\alpha \setminus \alpha$. 
	Then there exists an uncountable set $Z$ such that for all 
	$\alpha < \beta$ in $Z$, there exists a condition $(W,H)$ 
	such that $(W,H)$ extends $(T^\alpha,F^\alpha)$ and $(T^\beta,F^\beta)$ 
	and $x^\alpha <_W x^\beta$.
\end{thm}

\begin{proof}
	By extending further if necessary using Lemmas 5.4, 5.7, and 5.11, 
	we may assume that for all $\alpha < \omega_1$, 
	$\alpha \in \h[T^\alpha]$, $T^\alpha$ is normal, and $0 \in \dom(F^\alpha)$ 
	(the purpose of the last assumption is to ensure that different $F^\alpha$'s have 
	some common element in their domains). 
	We abbreviate the tree ordering on $T^\alpha$ as $<_\alpha$.

	By a standard thinning out argument, 
	we can fix sets $Z_0$, $A$, $T$, and functions $f_{\alpha,\beta}$ 
	and $g_{\alpha,\beta}$ for all $\alpha < \beta$ in $Z_0$ satisfying:
	\begin{enumerate}
	\item $Z_0$ is an uncountable subset of the set of $\xi < \omega_1$ such that 
	$\omega \cdot \xi = \xi$;
	\item $A \subseteq \omega_2$ is finite;
	\item $T$ is a standard finite tree;
	\item for all $\alpha < \beta$ in $Z_0$:
	\begin{enumerate}
		\item $T^\alpha \res \alpha = T^\beta \res \beta = T$ and 
		$T^\alpha \subseteq \beta$;
		\item $f_{\alpha,\beta}$ is an isomorphism from 
		$(T^\alpha,<_\alpha)$ to $(T^\beta,<_\beta)$ which is the identity function on $T$;
		\item $f_{\alpha,\beta}(x^\alpha) = x^\beta$;
		\item $\dom(F^\alpha) \cap \dom(F^\beta) = A$;
		\item $g_{\alpha,\beta} : \dom(F^\alpha) \to \dom(F^\beta)$ is a bijection 
		which is the identity function on $A$;
		\item for all $\tau \in \dom(F^\alpha)$, $m \in \{ -1, 1 \}$, 
		and $x \in T^\alpha$, 
		$$
		x \in \dom(F^\alpha(\tau)^m) \ \Longleftrightarrow \ 
		f_{\alpha,\beta}(x) \in \dom(F^\beta(g_{\alpha,\beta}(\tau))^{m}),
		$$
		and in that case, 
		$$
		f_{\alpha,\beta}(F^\alpha(\tau)^m(x)) = 
		F^\beta(g_{\alpha,\beta}(\tau))^m(f_{\alpha,\beta}(x)).
		$$
	\end{enumerate}
	\end{enumerate}

	Note that by property 4f, 
	for all $\tau \in \dom(F^\alpha)$, 
	$F^\alpha(\tau) \res T = F^\beta(g_{\alpha,\beta}(\tau)) \res T$, and in particular, 
	for all $\tau \in A$, 
	$F^\alpha(\tau) \res T = F^\beta(\tau) \res T$. 
	Let $f_{\beta,\alpha}$ and $g_{\beta,\alpha}$ denote the inverse functions 
	of $f_{\alpha,\beta}$ and $g_{\alpha,\beta}$ respectively. 
	These two inverse functions satisfy properties 4b, 4c, 4e, and 4f 
	when $\alpha$ and $\beta$ are switched.

	By the special property of $\rho$, fix an uncountable set $Z \subseteq Z_0$ 
	such that for all $\alpha < \beta$ in $Z$, $\alpha$ is greater than 
	$\max \{ \rho(\nu,\xi) : \nu, \xi \in A \}$, and for all 
	$\zeta \in \dom(F^\alpha) \setminus \dom(F^\beta)$, 
	for all $\tau \in \dom(F^\beta) \setminus \dom(F^\alpha)$, and 
	for all $\gamma \in \dom(F^\alpha) \cap \dom(F^\beta) = A$, 
	$$
	\rho(\zeta,\tau) \ge 
	\max \{ \min \{ \rho(\zeta,\gamma), \rho(\tau,\gamma) \}, \max(\h[T]) \}.
	$$
	Fix $\alpha < \beta$ in $Z$. 
	Note that since $0 \in \dom(F^\alpha) \cap \dom(F^\beta)$, 
	for all $\zeta$ and $\tau$ as above, $\rho(\zeta,\tau) \ge \max(\h[T])$. 
	The goal for the rest of the proof 
	is to find a condition $(W,H)$ which extends the conditions 
	$(T^\alpha,F^\alpha)$ and $(T^\beta,F^\beta)$ such that $x^\alpha <_W x^\beta$.
	
	\underline{Claim 1:} $f_{\alpha,\beta}(x^\alpha \res \alpha) = x^\beta \res \beta$.
	 
	\emph{Proof.} Since $f_{\alpha,\beta}$ is an isomorphism, 
	$x^\alpha \res \alpha \le_\alpha x^\alpha$ 
	implies that $f_{\alpha,\beta}(x^\alpha \res \alpha) \le_\beta 
	f_{\alpha,\beta}(x^\alpha) = x^\beta$. 
	Since $f_{\alpha,\beta}$ maps level $\alpha$ of $T^\alpha$ 
	onto level $\beta$ of $T^\beta$, it follows that 
	$f_{\alpha,\beta}(x^\alpha \res \alpha) = x^\beta \res \beta$. \qedsymbol

 	Let $X_\alpha$ be the closure of the singleton $\{ x^\alpha \res \alpha \}$ under 
	the functions in $\{ F^\alpha(\tau)^m : \tau \in A, \ m \in \{ -1, 1 \} \}$, and let 
	$X_\beta$ be the closure of $\{ x^\beta \res \beta \}$ under the functions in 
	$\{ F^\beta(\tau)^m : \tau \in A, \ m \in \{ -1, 1 \} \}$.

	\underline{Claim 2:} There exists an injective sequence 
	$\langle z_0,\ldots,z_{n-1} \rangle$ which lists $X_\alpha$ satisfying that 
	for any $l < n$, there exists a decreasing sequence $i_0 > \ldots > i_q$ such that 
	$i_0 = l$, $i_q = 0$, and for all $j < q$, 
	there exists some $(m,\zeta) \in \{ -1, 1 \} \times A$ 
	such that $F^\alpha(\zeta)^m(z_{i_j}) = z_{i_{j+1}}$. 
	Moreover, $X_\beta = \{ f_{\alpha,\beta}(z_i) : i < n \}$, so 
	$f_{\alpha,\beta} \res X_\alpha$ is a bijection between $X_\alpha$ and $X_\beta$.

	\emph{Proof.} We construct the sequence $\langle z_0,\ldots,z_{n-1} \rangle$ by induction. 
	Let $z_0 = x_\alpha \res \alpha$. 
	Now assume that $\langle z_0,\ldots,z_p \rangle$ is defined as required. 
	Apply one at a time the functions in 
	$\{ F^\alpha(\tau)^m : \tau \in A, \ m \in \{ -1, 1 \} \}$ 
	to the members of $\{ z_0,\ldots,z_p \}$ until we obtain a new element not in 
	$\{ z_0,\ldots,z_{p} \}$, 
	which we denote by $z_{p+1}$. 
	If this process does not result in any new element, then we are done and in that case 
	$X_\alpha = \{ z_0,\ldots,z_p \}$. 
	This completes the construction. 
	The first property described of this sequence can be easily proven by induction, and 
	the statement about $X_\beta$ follows from Claim 1 
	and property 4f. \qedsymbol

	\underline{Claim 3:} If $a$ and $b$ are distinct elements of $X_\alpha$, then 
	there exists an injective sequence $\langle x_0,\ldots,x_{k-1} \rangle$ 
	consisting of elements of $X_\alpha$ such that $k \ge 2$, $x_0 = a$, $x_{k-1} = b$, and 
	for all $i < k-1$ there exists some $(m,\tau) \in \{ -1, 1 \} \times A$ 
	such that $F^\alpha(\tau)^{m}(x_i) = x_{i+1}$.

	\emph{Proof.} Fix a sequence $\langle z_0,\ldots,z_{n-1} \rangle$ 
	as described in Claim 2. 
	Fix $l_a, l_b < n$ such that $a = z_{l_a}$ and $b = z_{l_b}$. 
	For each $c \in \{ a, b \}$ fix a decreasing sequence 
	$i^c_0 > \ldots > i^c_{q_c}$ such that 
	$i^c_0 = l_c$, $i^c_{q_c} = 0$, and for all $j < q_c$, 
	there exists some $(m,\zeta) \in \{-1,1\} \times A$ 
	such that $F^\alpha(\zeta)^m(z_{i^c_j}) = z_{i^c_{j+1}}$.
	
	Now consider the concatenation of the sequence 
	$\langle z_{i^a_0},\ldots,z_{i^a_{q_a}} \rangle$ 
	with the reverse of the sequence 
	$\langle z_{i^b_0},\ldots,z_{i^b_{q_b-1}} \rangle$. 
	This sequence starts at $a$, ends at $b$, and each of its elements has some 
	relation to its adjacent elements with respect to $\{ F^\alpha(\tau) : \tau \in A \}$. 
	So we are done provided that this concatenated sequence is injective. 
	If it is not, then we adjust it by deleting repetitions one at a time 
	going from left to right: 
	each time we encounter a subsequence of our current sequence of 
	the form $\langle d, \ldots, d \rangle$, 
	chosen as big as possible so that the last member of this subsequence 
	is the last occurrence of $d$ in the current sequence, 
	remove the elements of this subsequence after its first element. 
	Continue in this manner moving from left to right until 
	all repetitions are deleted and we obtain 
	an injective sequence. 
	Note that after every step of this process, the adjusted sequence still starts with $a$ 
	and ends with $b$, and adjacent elements are still related as required.  \qed
	
	\underline{Claim 4:} If $a$ and $b$ are distinct elements of $X_\alpha$ and there exists 
	some $\tau \in \dom(F^\alpha) \setminus A$ and $m \in \{ -1, 1 \}$ 
	such that $F^\alpha(\tau)^m(a) = b$, then there exists some 
	$\sigma \in A$ and $l \in \{ -1, 1 \}$ such that 
	$F^\alpha(\sigma)^l(a) = b$.
	
	\emph{Proof.} Apply Claim 3 to fix an 
	injective sequence $\langle x_0,\ldots,x_{k-1} \rangle$ 
	consisting of elements of $X_\alpha$ such that 
	$k \ge 2$, $x_0 = a$, $x_{k-1} = b$, and 
	for all $i < k-1$ there exists some $(m,\tau) \in \{ -1, 1 \} \times A$ 
	such that $F^\alpha(\tau)^{m}(x_i) = x_{i+1}$. 
	Assume that for some $\tau \in \dom(F^\alpha) \setminus A$ and $m \in \{ -1, 1 \}$, 
	$F^\alpha(\tau)^m(a) = b$. 
	If $k > 2$, then $\langle x_0,\ldots,x_{k-1}, x_0 \rangle$ is a loop, 
	contradicting that $F^\alpha(\tau)$ is 
	$\rho$-separated on $T^\alpha_\alpha$. 
	Hence, $k = 2$, so $a = x_0$ and $b = x_1$, 
	which easily implies the conclusion of the claim. \qedsymbol

	To prepare for the amalgamation of the conditions $(T^\alpha,F^\alpha)$ and 
	$(T^\beta,F^\beta)$, we make two preliminary steps. 
	First, we extend the tree $T^\alpha$ to a tree $(T^\alpha)^+$, and secondly, we 
	extend the condition $((T^\alpha)^+,F^\alpha)$ to a condition $(U,G)$. 
	For the first step, 
	for each $y \in T^\beta_\beta \setminus X_\beta$ we add to $T^\alpha$ a chain $C_y$ 
	above $y \res \max(\h[T])$, disjoint from $T^\alpha$, consisting of elements of 
	every possible height in $\h[T^\alpha] \setminus \alpha$. 
	Moreover, we arrange that any two such chains are disjoint. 
	Let $(T^\alpha)^+$ be the tree thus formed. 
	It is routine to check that $(T^\alpha)^+$ is a standard finite tree which is 
	normal and satisfies that $\h[(T^\alpha)^+] = \h[T^\alpha]$. 
	By Lemma 5.2, 
	$((T^\alpha)^+,F^\alpha)$ is a condition extending $(T^\alpha,F^\alpha)$. 
	Note that $(T^\alpha)^+ \res \alpha = T$ and 
	for all $\tau \in \dom(F^\alpha)$, the domain and range of 
	$F^\alpha(\tau)$ is disjoint from $(T^\alpha)^+ \setminus T^\alpha$.

	Applying the fact that $T^\alpha$ is normal, 
	fix an element $z^\alpha$ 
	on the top level of $T^\alpha$ such that $x^\alpha \le_\alpha z^\alpha$. 
	Since $\{ \rho(\nu,\xi) : \nu, \xi \in A \} \subseteq \alpha$, 
	it follows by Lemma 4.7(1) 
	that $\{ F^\alpha(\tau) : \tau \in A \}$ is separated on $X_\alpha$. 
	Applying Corollary 6.4, fix $(U,G) \in \p$ and 
	$X_\alpha^+ \subseteq U_{\max(\h[T^\alpha])}$ 
	satisfying:
	\begin{itemize}
		\item $(U,G) \le ((T^\alpha)^+,F^\alpha)$;
		\item $\h[T^\alpha] = \h[U]$ and $\dom(F^\alpha) = \dom(G)$;
		\item $U \setminus (T^\alpha)^+ \subseteq \bigcup \{ \Succ_U(x) : x \in X_\alpha \}$;
		\item for all $\tau \in \dom(G)$ and for all 
		$z \in \dom(G(\tau)) \setminus \dom(F^\alpha(\tau))$, both $z$ and 
		$G(\tau)(z)$ are in $\bigcup \{ \Succ_U(x) : x \in X_\alpha \}$;
		\item $X_\alpha^+$ has unique drop-downs to $\alpha$, $X_\alpha^+ \res \alpha = X_\alpha$, 
		and $z^\alpha \in X_\alpha^+$;	
		\item for all $\tau \in A$, $X_\alpha$ and $X_\alpha^+$ are $G(\tau)$-consistent.
	\end{itemize}
	Note that $U \res \alpha = T$ 
	and for all $\tau \in \dom(G)$, $G(\tau) \res T = F^\alpha(\tau) \res T$.

	To deal with the complexities of what follows, we split 
	the elements of $U \setminus \alpha$ into three disjoint sets. 
	Let $\mathcal{C}$ denote the set of elements belonging to some chain $C_y$, 
	where $y \in T^\beta_\beta \setminus X_\beta$; 
	in other words, $\mathcal{C} = (T^\alpha)^+ \setminus T^\alpha$. 
	Observe that for all $\tau \in \dom(G)$, the domain and range of 
	$G(\tau)$ is disjoint from $\mathcal{C}$. 
	Let $\mathcal{S}$ denote the set of $z \in U \setminus \alpha$ such that 
	for some $y \in X_\alpha^+$, $z = y \res \delta$ 
	for some $\delta \in \h[U] \setminus \alpha$. 
	In other words, 
	$\mathcal{S} = 
	\bigcup \{ X_\alpha^+ \res \delta : \delta \in \h[U] \setminus \alpha \}$. 
	Finally, let $\mathcal{D}$ denote the set of elements in 
	$U \setminus \alpha$ which are not in $\mathcal{C}$ or $\mathcal{S}$.

	Recall that $X_\alpha^+$ has unique drop-downs to $\alpha$, 
	$X_\alpha^+ \res \alpha = X_\alpha$, and 
	for all $\tau \in A$, $X_\alpha$ and $X_\alpha^+$ are $G(\tau)$-consistent. 
	It follows by Lemma 4.2(1) that for all $\delta \in \h[U] \setminus \alpha$, 
	$\mathcal{S} \cap U_\delta = X_\alpha^+ \res \delta$ 
	has unique drop-downs to $\alpha$, 
	$(\mathcal{S} \cap U_\delta) \res \alpha = X_\alpha$, and for all 
	$\tau \in A$, $X_\alpha$ and $\mathcal{S} \cap U_\delta$ are $G(\tau)$-consistent.

	\underline{Claim 5:} If $a$ and $b$ are distinct elements 
	of $\mathcal S$ of the same height, 
	then there exists an injective sequence $\langle y_0,\ldots,y_{k-1} \rangle$ 
	consisting of elements of $\mathcal S$ such that 
	$k \ge 2$, $y_0 = a$, $y_{k-1} = b$, and 
	for all $i < k-1$ there exists some $(m,\tau) \in \{ -1, 1 \} \times A$ 
	such that $G(\tau)^{m}(y_i) = y_{i+1}$.

	\emph{Proof.} Let $\delta$ be the height of $a$ and $b$. 
	Apply Claim 3 to $a \res \alpha$ and $b \res \alpha$ 
	to obtain an injective sequence 
	$\langle x_0,\ldots,x_{k-1} \rangle$ consisting 
	of elements of $X_\alpha$ such that 
	$k \ge 2$, $x_0 = a \res \alpha$, $y_{k-1} = b \res \alpha$, and 
	for all $i < k-1$ there exists some $(m,\tau) \in \{ -1 , 1 \} \times A$ 
	such that $F^\alpha(\tau)^{m}(x_i) = x_{i+1}$, that is, 
	$G(\tau)^m(x_i) = x_{i+1}$. 
	For each $0 < i < k$, let $y_i$ be the unique element of 
	$\mathcal S \cap U_\delta$ above $x_i$. 
	Note that $y_0 = a$ and $y_{k-1} = b$. 
	For all $\tau \in A$, $X_\alpha$ and $\mathcal{S} \cap U_\delta$ 
	are $G(\tau)$-consistent, so the desired conclusion clearly holds. \qedsymbol

	Now we begin the definition of the condition $(W,H)$. 
	This is done in two steps: first we construct $W$, 
	and then we construct $H$.
	
	\emph{Step 1:} Constructing $W$.
	
	We amalgamate the trees $U$ and $T^\beta$ 
	into a tree $W$ with underlying set $U \cup T^\beta$ as follows. 
	Since $U \res \alpha = T^\beta \res \beta = T$, 
	it suffices to specify for each $y \in T^\beta_\beta$ 
	an immediate predecessor $y^-$ of $y$ in $U_{\max(\h[U])}$ such that 
	$y \res_{T^\beta} \max(\h[T]) = y^{-} \res_U \max(\h[T])$. 
	Moreover, we can ensure that $W$ a simple extension of $T^\beta$ by arranging that 
	the function $y \mapsto y^- \res \alpha$ is injective. 
	For each $y \in T^\beta_\beta \setminus X_\beta$, let $y^-$ be the 
	top element of the chain $C_y$. 
	For each $z \in X_\beta$, let $z^-$ be 
	the unique element of $X_\alpha^+$ which is above $f_{\beta,\alpha}(z)$ in $U$. 
	It is easy to check that this works.

	Note that by definition, for all $y \in X_\beta$, 
	$y \res_W \alpha = f_{\beta,\alpha}(y)$. 
	Now $z^\alpha$ is the unique element of $X_\alpha^+$ which is 
	greater than or equal to $x^\alpha \res \alpha$ in $U$. 
	By Claim 1, 
	$x^\alpha \res \alpha = f_{\beta,\alpha}(x^\beta \res \beta)$. 
	So $(x^\beta \res \beta)^- = z^\alpha$. 
	Therefore, $x^\alpha <_W x^\beta$ as desired.
	
	Note that the elements of $\mathcal{D}$ do not have anything above them in 
	$W \setminus \beta$. 
	For each $z \in \mathcal{C}$, $z$ is in the chain $C_y$ for some 
	$y \in T^\beta_\beta \setminus X_\beta$, and $z <_W y$ by definition. 
	For each $z \in \mathcal S$, $z <_W f_{\alpha,\beta}(z \res \alpha)$. 
	For every $z \in \mathcal C \cup \mathcal S$, let $z^+$ be the unique member 
	of $W_\beta$ above $z$. 
	Note that $z \in \mathcal C$ iff $z^+ \in W_\beta \setminus X_\beta$, and 
	$z \in \mathcal S$ iff $z^+ \in X_\beta$. 
	Also, if $z \in \mathcal S$ 
	then $z \res \alpha = f_{\beta,\alpha}(z^+)$, that is, 
	$f_{\alpha,\beta}(z \res \alpha) = z^+$.

	For each $\tau \in \dom(F^\beta)$, let $\bar{F}^\beta(\tau)$ be 
	the downward closure of $F^\beta(\tau)$ in $W$. 
	Since $W$ is a simple extension of $T^\beta$, it follows by Lemma 3.8 that 
	$\bar{F}^\beta(\tau)$ is a standard function on $W$ 
	and $\bar{F}^\beta(\tau) \res T^\beta = F^\beta(\tau)$. 
	In particular, $\bar{F}^\beta(\tau) \res T = F^\beta(\tau) \res T$. 
	Observe that the domain and range of $\bar{F}^\beta(\tau)$ 
	is disjoint from $\mathcal D$. 
	Also, 
	$$
	\forall x, y \in \mathcal C \cup \mathcal S \ 
	(\bar{F}^\beta(\tau)(x) = y \ \Longleftrightarrow \ F^\beta(\tau)(x^+) = y^+).
	$$

	\underline{Claim 6:} For all $\tau \in A$, 
	$\bar{F}^\beta(\tau) \res T = G(\tau) \res T$. 
	
	\emph{Proof.} We have that 
	$\bar{F}^\beta(\tau) \res T = F^\beta(\tau) \res T = 
	F^\alpha(\tau) \res T = G(\tau) \res T$ (where property 4f is used for the 
	second equality). \qedsymbol

	\underline{Claim 7:} For all $(m,\tau) \in \{ -1,1 \} \times A$, 
	$\mathcal{C}$ and $\mathcal{S}$ 
	are both closed under $\bar{F}^\beta(\tau)^m$.

	\emph{Proof.} Assume that $\bar{F}^\beta(\tau)^m(x) = y$, where 
	$x \in \mathcal{C} \cup \mathcal{S}$. 
	Since the domain and range of $\bar{F}^\beta(\tau)$ are disjoint from $\mathcal D$, 
	$y \in \mathcal C \cup \mathcal S$. 
	So $F^\beta(\tau)^m(x^+) = y^+$ by the above.
	If $x \in \mathcal S$, then $x^+ \in X_\beta$. 
	Since $X_\beta$ is closed under $F^\beta(\tau)^m$, $y^+ \in X_\beta$, 
	and hence $y \in \mathcal S$. 
	If $x \in \mathcal C$, then $x^+ \in W_\beta \setminus X_\beta$. 
	But $X_\beta$ is closed under $F^\beta(\tau)^{-m}$, 
	so $y^+ \in W_\beta \setminus X_\beta$. 
	Thus, $y \in \mathcal C$. \qedsymbol

	\underline{Claim 8:} For all $(m,\tau) \in \{ -1, 1 \} \times A$, 
	$\bar{F}^\beta(\tau)^m \res \mathcal{S} = G(\tau)^m \res \mathcal{S}$.

	\emph{Proof.} Suppose that $x \in \mathcal S$ and $\bar{F}^\beta(\tau)^m(x) = y$. 
	Then $y \in \mathcal S$ by Claim 7, $y^+ \in X_\beta$, and 
	$F^\beta(\tau)^m(x^+) = y^+$. 
	We also have that $y \res \alpha = f_{\beta,\alpha}(y^+) = 
	f_{\beta,\alpha}(F^\beta(\tau)^m(x^+)) = 
	F^\alpha(\tau)^m(f_{\beta,\alpha}(x^+)) = F^\alpha(\tau)^m(x \res \alpha) = 
	G(\tau)^m(x \res \alpha)$. 
	By $G(\tau)$-consistency, it follows that $G(\tau)^m(x) = y$. 

	On the other hand, assume that $x \in \mathcal S$ and $G(\tau)^m(x) = y$. 
	Then $F^\alpha(\tau)^m(x \res \alpha) = y \res \alpha$. 
	Since $X_\alpha$ is closed under $F^\alpha(\tau)^m$, $y \res \alpha \in X_\alpha$. 
	Let $y'$ be the unique element of $\mathcal S$ 
	above $y \res \alpha$ with the same height as $x$. 
	Then by $G(\tau)$-consistency, $G(\tau)^m(x) = y'$. 
	Hence, $y = y'$ and so $y \in \mathcal S$. 
	By property 4f, 
	$y^+ = f_{\alpha,\beta}(y \res \alpha) = 
	f_{\alpha,\beta}(F^\alpha(\tau)^m(x \res \alpha)) = 
	F^\beta(\tau)^m(f_{\alpha,\beta}(x \res \alpha)) = 
	F^\beta(\tau)^m(x^+)$. 
	Hence, $y = \bar{F}^\beta(\tau)^m(x)$. \qedsymbol
	
	The following two claims record for future reference observations which we 
	already made above.
	
	\underline{Claim 9:} For all $\tau \in \dom(G)$, the domain and range of $G(\tau)$ 
	are disjoint from $\mathcal C$.
	
	\underline{Claim 10:} For all $\tau \in \dom(F^\beta)$, the domain and range 
	of $\bar{F}^\beta(\tau)$ are disjoint from $\mathcal D$.

	By Lemma 5.2, $(W,G) \in \p$ and $(W,G) \le (U,G)$. 
	By Lemma 5.3, $(W,\bar F^\beta) \in \p$ and $(W,\bar F^\beta) \le (T^\beta,F^\beta)$. 
	So it suffices to construct $H$ so that 
	$(W,H) \in \p$ and $(W,H)$ extends both $(W,G)$ and $(W,\bar F^\beta)$.

	\emph{Step 2:} Constructing $H$.

	Let the domain of $H$ be equal to $\dom(G) \cup \dom(F^\beta)$. 
	If $\tau \in \dom(G) \setminus \dom(F^\beta)$, then let $H(\tau) = G(\tau)$. 
	If $\tau \in \dom(F^\beta) \setminus \dom(G)$, then let 
	$H(\tau) = \bar{F}^\beta(\tau)$. 
	Now suppose that $\tau \in \dom(G) \cap \dom(F^\beta)$. 
	Then $\tau \in \dom(F^\alpha) \cap \dom(F^\beta) = A$. 
	We claim that for all $x \in \dom(\bar{F}^\beta(\tau)) \cap \dom(G(\tau))$, 
	$\bar{F}^\beta(\tau)(x) = G(\tau)(x)$. 
	It then easily follows that $H(\tau) = \bar{F}^\beta(\tau) \cup G(\tau)$ 
	is a strictly increasing, level preserving, downwards closed partial 
	function from $W$ to $W$ with no fixed-points other than $0$. 
	So let $x \in \dom(\bar{F}^\beta(\tau)) \cap \dom(G(\tau))$. 
	By Claim 6, we can assume that $x \notin T$. 
	Also, $x \in \dom(G(\tau))$ implies that $x \in U$. 
	By Claims 9 and 10, $x \in \mathcal S$. 
	By Claim 8, $\bar{F}^\beta(\tau)(x) = G(\tau)(x)$.

	Now we show that $H(\tau)$ is injective. 
	Since $\bar{F}^\beta(\tau)$ and $G(\tau)$ are each injective, it is enough to 
	show that for all $x \in \dom(\bar{F}^\beta(\tau)) \setminus \dom(G(\tau))$ 
	and for all $y \in \dom(G(\tau)) \setminus \dom(\bar{F}^\beta(\tau))$, 
	$\bar{F}^\beta(\tau)(x) \ne G(\tau)(y)$. 
	As $\bar{F}^\beta(\tau) \res T = G(\tau) \res T$ by Claim 6 and 
	$H(\tau)$ is level preserving, it suffices 
	to consider $x$ and $y$ which are both 
	in $U_\delta$ for some $\delta \in \h[U] \setminus \alpha$. 
	Suppose for a contradiction that 
	$\bar{F}^\beta(\tau)(x) = z = G(\tau)(y)$. 
	By Claims 9 and 10, $z \in \mathcal S$. 
	By Claim 7, $x$ is in $\mathcal S$. 
	By Claim 8, $x$ is in the domain of $G(\tau)$, which is a contradiction. 
	This completes the proof that $H(\tau)$ is injective and hence is a standard 
	function on $W$.

	We need two more claims.
		
	\underline{Claim 11:} If $a, b \in \mathcal{S}$, $\tau \in \dom(H)$, 
	and $H(\tau)(a) = b$, then there exists some 
	$(m,\zeta) \in \{ -1, 1 \} \times \dom(F^\beta)$ 
	such that $F^\beta(\zeta)^m(a^+) = b^+$.

	\emph{Proof.} The equation $H(\tau)(a) = b$ means that either 
	$\bar{F}^\beta(\tau)(a) = b$ or $G(\tau)(a) = b$. 
	In the former case, $F^\beta(\tau)(a^+) = b^+$ and we are done. 
	Suppose that $G(\tau)(a) = b$. 
	By Claim 8, we may assume that $\tau \notin A$, 
	otherwise we are in the case just considered. 
	Since $G(\tau)(a) = b$, it follows that 
	$G(\tau)(a \res \alpha) = b \res \alpha$, that is, 
	$F^\alpha(\tau)(a \res \alpha) = b \res \alpha$. 
	By Claim 4, there exists some $(m,\xi) \in \{ -1, 1 \} \times A$ 
	such that $F^\alpha(\xi)^m(a \res \alpha) = b \res \alpha$. 
	By property 4f, 
	$F^\beta(\xi)^m(a^+) = F^\beta(\xi)^m(f_{\alpha,\beta}(a \res \alpha)) = 
	f_{\alpha,\beta}(F^\alpha(\xi)^m(a \res \alpha)) = f_{\alpha,\beta}(b \res \alpha) = b^+$.
	\qedsymbol
	
	\underline{Claim 12:} If $a, b \in \mathcal{S}$, $\gamma \in \dom(F^\beta)$, 
	and $\bar{F}^\beta(\gamma)(a) = b$, then there exists some 
	$(m,\xi) \in \{ -1, 1 \} \times \dom(G)$ such that $G(\xi)^m(a) = b$.

	\emph{Proof.} By Claim 8, we are done if $\gamma \in A$. 
	So assume that $\gamma \in \dom(F^\beta) \setminus \dom(G)$. 
	Then neither $\gamma$ nor $g_{\beta,\alpha}(\gamma)$ are in $A$. 
	We have that $F^\beta(\gamma)(a^+) = b^+$, 
	so $F^\alpha(g_{\beta,\alpha}(\gamma))(a \res \alpha) = 
	F^\alpha(g_{\beta,\alpha}(\gamma))(f_{\beta,\alpha}(a^+)) = 
	f_{\beta,\alpha}(F^\beta(\gamma)(a^+)) = 
	f_{\beta,\alpha}(b^+) = b \res \alpha$. 
	By Claim 4, there exists some $(m,\xi) \in \{ -1, 1 \} \times A$ such that 
	$G(\xi)^m(a \res \alpha) = b \res \alpha$. 
	By $G(\xi)$-consistency, $G(\xi)^m(a) = b$. \qedsymbol

	We have now constructed $(W,H)$. 
	To complete the proof, 
	we need to show that $(W,H) \in \p$ and 
	$(W,H)$ extends both $(W,G)$ and $(W,\bar{F}^\beta)$. 
	For proving the former, it suffices to show that 
	for all $\delta \in \h[W]$, 
	$H$ is $\rho$-separated on $W_\delta$. 
	Once we know that $(W,H) \in \p$, 
	in order to prove that $(W,H)$ extends $(W,G)$ and $(W,\bar{F}^\beta)$, 
	it suffices to verify Definition 5.1(c) in each case.
	
	Fix $\delta \in \h[W]$ and we prove that 
	$H$ is $\rho$-separated on $W_\delta$ using Proposition 4.10 
	(Characterization of $\rho$-Separation). 
	Note that if $\delta \ge \beta$, 
	then since $W \setminus \beta = T^\beta \setminus \beta$ 
	and $H$ and $F^\beta$ are the same on $W \setminus \beta$, we are done 
	because $(T^\beta,F^\beta)$ is a condition. 
	So assume that $\delta < \beta$.
	
	\emph{Verifying property (1) of Proposition 4.10 (Characterization of $\rho$-Separation):}
	
	Suppose that $x, y \in W_\delta$  
	and $(m_0,\tau_0)$ and $(m_1,\tau_1)$ 
	are distinct pairs in $\{ -1, 1 \} \times \dom(H)$ satisfying that 
	$H(\tau_0)^{m_0}(x) = y$ and $H(\tau_1)^{m_1}(x) = y$. 
	We prove that $\rho(\tau_0,\tau_1) \ge \delta$. 
	Since $(W,G)$ and $(W,\bar{F}^\beta)$ are conditions, without loss of generality 
	we may assume that $G(\tau_0)^{m_0}(x) = y$ and 
	$\bar{F}^\beta(\tau_1)^{m_1}(x) = y$.

	First, suppose that $\delta < \alpha$. 	
	Then $W_\delta = T^\alpha_\delta$. 
	Recall that $G$ is the same as $F^\alpha$ on $T$ and 
	$\bar{F}^\beta$ is the same as $F^\beta$ on $T$. 
	So $F^\alpha(\tau_0)^{m_0}(x) = y$ and $F^\beta(\tau_1)^{m_1}(x) = y$. 
	If for some $k < 2$, $\tau_k \in A$, then 
	$F^\alpha(\tau_k) \res T = F^\beta(\tau_k) \res T$ so we are done 
	since $(T^\alpha,F^\alpha)$ and $(T^\beta,F^\beta)$ are conditions. 
	So assume that $\tau_0 \in \dom(F^\alpha) \setminus \dom(F^\beta)$ and 
	$\tau_1 \in \dom(F^\beta) \setminus \dom(F^\alpha)$. 
	Then by the choice of $\alpha$ and $\beta$, 
	$\rho(\tau_0,\tau_1) \ge \max(\h[T]) \ge \delta$.

	Secondly, assume that $\alpha \le \delta < \beta$. 
	By Claims 9 and 10, 
	$x$ and $y$ are in $\mathcal S$. 
	By Claim 8, we may assume that neither $\tau_0$ nor $\tau_1$ are in $A$, 
	for otherwise we are done since $(W,G)$ and $(W,\bar{F}^\beta)$ are conditions. 
	So $\tau_0 \in \dom(G) \setminus \dom(F^\beta)$ and 
	$\tau_1 \in \dom(F^\beta) \setminus \dom(G)$. 
	In particular, $\tau_0 \notin A$. 
	As $F^\alpha(\tau_0)^{m_0}(x \res \alpha) = 
	G(\tau_0)^{m_0}(x \res \alpha) = y \res \alpha$, 
	by Claim 4 there exists some $(n,\sigma) \in \{ -1, 1 \} \times A$ 
	such that $G(\sigma)^{n}(x \res \alpha) = y \res \alpha$. 
	By $G(\sigma)$-consistency, $G(\sigma)^n(x) = y$. 
	Since $(W,G)$ is a condition, $\rho(\tau_0,\sigma) \ge \delta$. 
	By Claim 8, $\bar{F}^\beta(\sigma)^{n}(x) = y$. 
	As $(W,\bar{F}^\beta)$ is a condition, 
	$\rho(\tau_1,\sigma) \ge \delta$. 
	Therefore,
	$$
	\rho(\tau_0,\tau_1) \ge 
	\min \{ \rho(\tau_0,\sigma), \rho(\tau_1,\sigma) \} \ge \delta.
	$$

	\emph{Verifying property (2) of Proposition 4.10 (Characterization of $\rho$-Separation):}
 
	Suppose for a contradiction 
	that there exists a loop $\langle a_0,\ldots,a_{n-1} \rangle$ of elements of $W_\delta$ 
	with respect to $H$. 
	So $n \ge 4$, the sequence $\langle a_0,\ldots,a_{n-2} \rangle$ is injective, 
	$a_0 = a_{n-1}$, and for all $i < n-1$ there exists 
	$(m_i,\tau_i) \in \{ -1, 1 \} \times \dom(H)$ such that 
	$H(\tau_i)^{m_i}(a_i) = a_{i+1}$. 
	For each $i < n$, either $G(\tau_i)^{m_i}(a_i) = a_{i+1}$ or 
	$\bar{F}^\beta(\tau_i)^{m_i}(a_i) = a_{i+1}$.

	First, assume that $\delta < \alpha$. 
	Since $G$ is equal to $F^\alpha$ on $T$ and $\bar{F}^\beta$ is equal to $F^\beta$ on $T$, 
	for each $i < n$, either $F^\alpha(\tau_i)^{m_i}(a_i) = a_{i+1}$ or 
	$F^\beta(\tau_i)^{m_i}(a_i) = a_{i+1}$. 
	Let $i < n$. 
	If $\tau_i \in A$, then $F^\alpha(\tau_i) \res T = F^\beta(\tau_i) \res T$, 
	so in either case, 
	$F^\alpha(\tau_i)^{m_i}(a_i) = a_{i+1}$. 
	Suppose that it is not the case that $F^\alpha(\tau_i)^{m_i}(a_i) = a_{i+1}$. 
	Then $F^\beta(\tau_i)^{m_i}(a_i) = a_{i+1}$. 
	By property 4f, $F^\alpha(g_{\beta,\alpha}(\tau_i))^{m_i}(a_i) = a_{i+1}$. 
	It follows that $\langle a_0,\ldots,a_{n-1} \rangle$ is a loop in $T^\alpha_\delta$ 
	with respect to $F^\alpha$, which contradicts that $(T^\alpha,F^\alpha)$ is a condition.

	Secondly, assume that $\alpha \le \delta < \beta$. 
	We consider four cases.

	Case 1: For all $i < n$, $a_i \in \mathcal{C} \cup \mathcal{S}$.
	Consider $i < n-1$. 
	Then either $\bar{F}^\beta(\tau_i)^{m_i}(a_i) = a_{i+1}$ or 
	$G(\tau_i)^{m_i}(a_i) = a_{i+1}$. 
	If either $a_i$ or $a_{i+1}$ is in $\mathcal C$, 
	then we are in the former case by Claim 9. 
	Hence, $F^\beta(\tau_i)^{m_i}(a_i^+) = a_{i+1}^+$. 
	Otherwise, both $a_i$ and $a_{i+1}$ are in $\mathcal S$. 
	By Claim 11, there exists some 
	$(m,\zeta) \in \{ -1, 1 \} \times \dom(F^\beta)$ such that 
	$F^\beta(\zeta)^m(a_i^+) = a_{i+1}^+$. 
	So $\langle a_0^+,\ldots,a_{n-1}^+ \rangle$ is a loop in $T^\beta_\beta$ 
	with respect to $F^\beta$, which contradicts 
	that $(T^\beta,F^\beta)$ is a condition.

	Case 2: For all $i < n$, $a_i \in \mathcal{D} \cup \mathcal{S}$. 
	Let $i < n-1$. 
	If either $a_i$ or $a_{i+1}$ is in $\mathcal D$, 
	then $G(\tau_i)^{m_i}(a_i) = a_{i+1}$ by Claim 10. 
	Otherwise, $a_i$ and $a_{i+1}$ are both in $\mathcal S$, 
	so by Claim 12 there exists some 
	$(m,\xi) \in \{ -1, 1 \} \times \dom(G)$ such that 
	$G(\xi)^{m}(a_i) = a_{i+1}$. 
	It follows that $\langle a_0,\ldots,a_{n-1} \rangle$ is a loop in $U_\delta$ 
	with respect to $G$, which contradicts that $(U,G)$ is a condition.
	
	Case 3: For all $i < n$, $a_i \in \mathcal{C} \cup \mathcal{D}$. 
	By Claims 9 and 10, 
	there are no relations between members of $\mathcal C$ and members 
	of $\mathcal D$ with respect to $H$. 
	So either for all $i < n$, 
	$a_i \in \mathcal C$, or for all $i < n$, $a_i \in \mathcal D$. 
	So we are in either Case 1 or Case 2, which were already handled.

	Case 4: The sequence $\langle a_0,\ldots,a_{n-1} \rangle$ contains at least one member 
	in each of $\mathcal C$, $\mathcal D$, and $\mathcal S$. 
	Since there are no relations between members of $\mathcal C$ and members 
	of $\mathcal D$ with respect to $H$, there do not exist 
	adjacent elements of the loop where one is in 
	$\mathcal C$ and the other is in $\mathcal D$. 
	By shifting the sequence if necessary, we may assume without loss of generality 
	that $a_0 \in \mathcal C$ and $a_1 \in \mathcal S$. 
	Since $n \ge 4$, we know that $a_0$, $a_1$, and $a_{n-2}$ are all distinct. 
	Let $k < n-1$ be largest such that $a_k$ is not in $\mathcal C$. 
	So every member of the sequence after $a_k$ is in $\mathcal C$, and in particular, 
	$a_{k+1} \in \mathcal C$, so $a_k \in \mathcal S$. 
	Moreover, $a_1$ and $a_k$ are different because the loop contains at least one 
	element which is in $\mathcal D$, and that member of $\mathcal D$ 
	must be between $a_1$ and $a_k$. 
	Apply Claim 5 to fix an injective sequence $\langle b_0,\ldots,b_{p-1} \rangle$ 
	of elements of $\mathcal S$ such that $p \ge 2$, $b_0 = a_1$, $b_{p-1} = a_k$, 
	and for all $j < p-1$ there exists $(m,\sigma) \in \{ -1, 1 \} \times A$ 
	such that $G(\sigma)^m(b_i) = b_{i+1}$. 

	Now consider the sequence 
	$$
	\langle a_0, a_1, b_1, \ldots, b_{p-2}, a_k, 
	a_{k+1},\ldots,a_{n-1} \rangle.
	$$
	Note that this sequence has length at least $4$, 
	all of the elements of this sequence are in $\mathcal C \cup \mathcal S$, 
	any two adjacent elements of this sequence are related with respect to $H$, 
	and this
	sequence minus the last element is injective. 
	Thus, we have a loop in $W_\delta$ with respect to $H$ consisting of 
	members of $\mathcal C \cup \mathcal S$. 
	So we are in Case 1, which was already handled.
	
	This completes the proof that $(W,H)$ is a condition. 

	It remains to prove that $(W,H)$ extends $(W,G)$ and $(W,\bar{F}^\beta)$. 
	In both cases, it suffices to verify Definition 5.1(c).

	\emph{Proving that $(W,H) \le (W,\bar{F}^\beta)$:}
	
	Suppose that $\gamma$ and $\tau$ are distinct elements of $\dom(F^\beta)$, 
	$x \in \dom(H(\gamma)) \cap \dom(H(\tau))$, and $H(\gamma)(x) = H(\tau)(x)$. 
	We prove that there exists some $z \in W$ such that $x \le_{W} z$ 
	and $\bar{F}^\beta(\gamma)(z) = \bar{F}^\beta(\tau)(z)$. 
	If $x \in W \setminus \beta$, then $H(\gamma) = F^\beta(\gamma)$ and 
	$H(\tau) = F^\beta(\tau)$, so we are done. 
	Assume that $x \in T$. 
	For each $\xi \in \{ \gamma, \tau \}$, 
	$H(\xi)(x)$ is equal to either $F^\alpha(\xi)(x)$ or $F^\beta(\xi)(x)$. 
	Moreover, by property 4f we are in both cases when $\xi \in A$. 
	However, since $\xi \in \dom(F^\beta)$, 
	if $H(\xi)(x) = F^\alpha(\xi)(x)$, then $\xi \in A$. 
	So no matter what, $H(\gamma)(x) = F^\beta(\gamma)(x)$ and 
	$H(\tau)(x) = F^\beta(\tau)(x)$, and we are done.

	Now assume that $x \in U_\delta$ for some $\alpha \le \delta < \beta$. 
	Let $y = H(\gamma)(x)$. 
	So also $y = H(\tau)(x)$. 
	Suppose that $x \in \mathcal C$. 
	Then by Claim 9, 
	$H(\gamma)(x) = \bar{F}^\beta(\gamma)(x)$ 
	and $H(\tau)(x) = \bar{F}^\beta(\gamma)(x)$. 
	Therefore, $F^\beta(\gamma)(x^+) = y^+ = F^\beta(\tau)(x^+)$ 
	and we are done. 
	If $x \in \mathcal D$, then by Claim 10, 
	$H(\gamma)(x) = G(\gamma)(x)$ and 
	$H(\tau)(x) = G(\tau)(x)$. 
	So $\gamma$ and $\tau$ are in $\dom(G) = \dom(F^\alpha)$, and hence are in $A$. 
	So $\rho(\gamma,\tau) < \alpha$.  
	Since $G$ is $\rho$-separated on $W_\delta$, 
	the equations $G(\gamma)(x) = y$ and $G(\tau)(x) = y$ imply 
	that $\rho(\gamma,\tau) \ge \delta \ge \alpha$, which is a contradiction.
	
	Finally, assume that $x \in \mathcal S$. 
	If $H(\gamma)(x) = \bar{F}^\beta(\gamma)(x)$ 
	and $H(\tau)(x) = \bar{F}^\beta(\tau)(x)$, then 
	$F^\beta(\gamma)(x^+) = y^+$ and $F^\beta(\tau)(x^+) = y^+$ and we are done. 
	Otherwise, without loss of generality, $H(\gamma)(x) = G(\gamma)(x)$ and it is not 
	the case that $H(\gamma)(x) = \bar{F}^\beta(\gamma)(x)$. 
	But $H(\gamma)(x) = G(\gamma)(x)$ implies that $\gamma \in A$, 
	and then Claim 8 implies $H(\gamma)(x) = \bar{F}^\beta(\gamma)(x)$, which is a contradiction.
	This completes the proof that $(W,H) \le (W,\bar{F}^\beta)$. 

	\emph{Proving that $(W,H) \le (W,G)$:}
	
	Suppose that $\gamma$ and $\tau$ are distinct elements of $\dom(G)$, 
	$x \in \dom(H(\gamma)) \cap \dom(H(\tau))$, and 
	$H(\gamma)(x) = H(\tau)(x)$. 
	Let $y = H(\gamma)(x)$. 
	We prove that there exists some $z \in W$ such that $x \le_{W} z$ 
	and $G(\gamma)(z) = G(\tau)(z)$. 
	If $H(\gamma)(x) = G(\gamma)(x)$ and $H(\tau)(x) = G(\tau)(x)$, then we are done. 
	So assume without loss of generality that $H(\gamma)(x) = \bar{F}^\beta(\gamma)(x)$ 
	and it is not the case that $H(\gamma)(x) = G(\gamma)(x)$. 
	It follows that $\gamma \in A$. 
	By Claim 6, we may assume that $x \notin T$. 
	Assume that $x \in W_\delta = T^\beta_\delta$ for some $\beta \le \delta$. 
	Then $H(\tau)(x) = F^\beta(\tau)(x)$ and also $\tau \in A$. 
	Since $F^\beta$ is $\rho$-separated on $T^\beta_\delta$, 
	$\rho(\gamma,\tau) \ge \delta \ge \beta$. 
	But $\gamma$ and $\tau$ are both in $A$, so $\rho(\gamma,\tau) < \alpha < \delta$, 
	which is a contradiction.

	Finally, assume that $x \in W_\delta = U_\delta$ for some $\alpha \le \delta < \beta$. 
	If $x \in \mathcal C$, then by Claim 9, 
	$H(\tau)(x) = \bar{F}^\beta(\tau)(x)$. 
	So $\gamma$ and $\tau$ are both in $A$, and hence $\rho(\gamma,\tau) < \alpha$. 
	But $\bar{F}^\beta(\gamma)(x) = y = \bar{F}^\beta(\tau)(x)$ implies that 
	$F^\beta(\gamma)(x^+) = y^+ = F^\beta(\tau)(x^+)$. 
	Since $F^\beta$ is $\rho$-separated on $T^\beta_\beta$, 
	$\rho(\gamma,\tau) \ge \beta$, which is a contradiction. 
	If $x \in \mathcal D$, then by Claim 10, 
	$H(\gamma)(x) = G(\gamma)(x)$, 
	which contradicts our assumption. 
	Finally, assume that $x \in \mathcal S$. 
	By Claim 8, $H(\gamma)(x) = \bar{F}^\beta(\gamma)(x) = G(\gamma)(x)$, 
	which again contradicts our assumption.
	\end{proof}

\begin{corollary}
	The forcing poset $\p$ is Knaster.
\end{corollary}

\begin{proof}
	Let $\{ (T^\alpha,F^\alpha) : \alpha < \omega_1 \}$ be a family of conditions. 
	Applying Lemma 5.4, without loss of generality we may assume that for all $\alpha < \omega_1$, 
	$\alpha \in \h[T^\alpha]$. 
	For each $\alpha < \omega_1$, fix some $x^\alpha \in T^\alpha_\alpha$. 
	Now apply Theorem 7.1 to find an uncountable set $Z \subseteq \omega_1$ 
	such that for all $\alpha < \beta$ in $Z$, 
	$(T^\alpha,F^\alpha)$ and $(T^\beta,F^\beta)$ are compatible.
\end{proof}

\begin{corollary}
	The forcing poset $\p$ forces that $T^{{\dot{G}_{\p}}}$ is Suslin.
\end{corollary}

\begin{proof}
	Suppose for a contradiction that some condition $p \in \p$ forces that there exists 
	an uncountable antichain of $T^{{\dot{G}_{\p}}}$. 
	Then we can find a sequence of $\p$-names 
	$\langle \dot x^\alpha : \alpha < \omega_1 \rangle$ for elements of $T^{{\dot{G}_{\p}}}$ 
	such that $p$ forces that 
	for each $\alpha < \omega_1$, $\h(\dot x^\alpha) \ge \alpha$, and 
	for all $\alpha < \beta < \omega_1$, $\dot x^\alpha$ and $\dot x^\beta$ 
	are incomparable.  
	For each $\alpha < \omega_1$, pick a condition $(T^\alpha,F^\alpha) \le p$ 
	and some $x^\alpha \in T^\alpha$ such that 
	$(T^\alpha,F^\alpha)$ forces that $\dot x^\alpha$ is 
	equal to $\check x^\alpha$. 
	By Theorem 7.1, there are $\alpha < \beta < \omega_1$ and a condition 
	$(W,H)$ extending $(T^\alpha,F^\alpha)$ and $(T^\beta,F^\beta)$ such that 
	$x^\alpha <_W x^\beta$. 
	But this contradicts that $(W,H)$ forces that $\dot x^\alpha$ and $\dot x^\beta$ 
	are incomparable.
\end{proof}

We have now completed the proof of the main theorem.

\bigskip

We close the article with a question.

\begin{question}
	Is it consistent that there exists a strongly non-saturated Aronszajn tree 
	and there does not exist a weak Kurepa tree? 
\end{question}

We remark that we do not know whether the forcing $\p$ itself always adds a Kurepa 
tree or a weak Kurepa tree.

\bigskip

\textbf{Acknowledgements:} 
The first author acknowledges support from the Simons Foundation under 
the Travel Support for Mathematicians gift 631279. 
We thank the referee for providing helpful suggestions.

\providecommand{\bysame}{\leavevmode\hbox to3em{\hrulefill}\thinspace}
\providecommand{\MR}{\relax\ifhmode\unskip\space\fi MR }
\providecommand{\MRhref}[2]{%
  \href{http://www.ams.org/mathscinet-getitem?mr=#1}{#2}
}
\providecommand{\href}[2]{#2}


\end{document}